\documentclass{article}[11pt,a4paper]

\usepackage[frenchb,english]{babel}
\usepackage{amsmath}
\usepackage{amsfonts}
\usepackage{amsthm}
\usepackage{amssymb}
\usepackage{bm}
\usepackage{verbatim}
\usepackage[active]{srcltx}

\usepackage[latin1]{inputenc}

\usepackage{enumerate}
\usepackage{multicol}

\usepackage[all,cmtip]{xy}

\newtheoremstyle{mystyle2}{}{}{}{2pt}{\scshape}{.}{ }{}
\newtheoremstyle{mystyle}{}{}{\slshape}{2pt}{\scshape}{.}{ }{}
\newtheoremstyle{etapestyle}{}{}{\itshape}{2em}{\sffamily}{:}{ }{\thmname{#1}}
\newtheoremstyle{definitionstyle}{}{}{}{2pt}{\bfseries}{.}{ }{}

\newtheorem{thm}{Th\'{e}or\`{e}me}[section]
\newtheorem{cor}[thm]{Corollaire}

\newtheorem{prop}[thm]{Proposition}

\newtheorem{lemme}[thm]{Lemme}

\newtheorem*{prop*}{Proposition}

\theoremstyle{mystyle2}

\theoremstyle{mystyle}

\theoremstyle{remark}
\newtheorem{rem}[thm]{Remarque}

\newtheorem{conventions}[thm]{Conventions}
\theoremstyle{etapestyle}
\newtheorem{etape1}{Etape 1 }
\newtheorem{etape2}{Etape 2 }

\theoremstyle{definitionstyle}
\newtheorem{defi}[thm]{D\'{e}finition}

 \DeclareMathOperator{\Spec}{Spec}

\DeclareMathOperator{\Ker}{Ker} 
\DeclareMathOperator{\Pic}{Pic}
\DeclareMathOperator{\Card}{Card}
\DeclareMathOperator{\Sing}{Sing}
\DeclareMathOperator{\Sym}{Sym}
\DeclareMathOperator{\Hilb}{Hilb}
\DeclareMathOperator{\Ima}{Im}

\title{Quelques espaces de modules d'intersections compl\`etes lisses qui sont quasi-projectifs}
\author{Olivier BENOIST\footnote{\'Ecole Normale Sup\'erieure, D\'epartement de Math\'ematiques et Applications,  
45 rue d'Ulm 75230 Paris Cedex 05 France, email : {obenoist@dma.ens.fr}}}
\date{}

\begin{document}

\maketitle
\renewcommand{\abstractname}{R\'esum\'e}
\begin{abstract} 
Pour certaines valeurs des degr\'es des \'equations, on montre, \`a l'aide de th\'eorie g\'eom\'etrique des invariants,
que l'espace de modules grossier des intersections compl\`etes lisses dans $\mathbb{P}^N$ est quasi-projectif. 
 \end{abstract}
   \selectlanguage{english}
\renewcommand{\abstractname}{Abstract}
\begin{abstract} 
For some values of the degrees of the equations,
we show, using geometric invariant theory, that the coarse moduli space of
smooth complete intersections in $\mathbb{P}^N$ is quasi-projective.\end{abstract}
\selectlanguage{frenchb}

\section{Introduction}\label{introqp}

\subsection{Quasi-projectivit\'e d'espaces de modules}

  La question de la quasi-projectivit\'e des espaces de modules de vari\'et\'es al\-g\'e\-briques a \'et\'e
r\'evolutionn\'ee par Mumford qui a d\'evelopp\'e pour l'\'etudier la th\'eorie g\'eom\'etrique des invariants.
Cette technique a permis \`a Mumford \cite{GIT} de montrer la quasi-projectivit\'e de l'espace de modules des courbes lisses,
puis \`a Knudsen \cite{KnudsenIII} et Gieseker et Mumford \cite{Mumfordstab} de montrer ind\'ependamment la projectivit\'e de l'espace
des modules des courbes stables.
En dimension sup\'erieure, les travaux de Viehweg \cite{Viehweg} montrent la quasi-projectivit\'e
des espaces de modules de vari\'et\'es lisses canoniquement polaris\'ees en caract\'eristique nulle. 

Une autre strat\'egie pour montrer la quasi-projectivit\'e d'un espace de modules, efficace quand
celui-ci est propre, a \'et\'e
d\'evelopp\'ee par Koll\'ar \cite{Kollarcomplete}. Elle devrait permettre de montrer la projectivit\'e de compactifications modulaires
des espaces de modules \'etudi\'es par Viehweg (voir \cite{livreKollar}).

Dans ces exemples, le fibr\'e canonique des vari\'et\'es consid\'er\'ees v\'erifie des propri\'et\'es de
positivit\'e. A contrario, on ne conna\^it pas d'\'enonc\'e g\'en\'eral sur la quasi-projectivit\'e des espaces de modules 
de vari\'et\'es de Fano.
Vu les exemples de Koll\'ar \cite{Kollarcex} d'espaces de modules non quasi-projectifs de vari\'et\'es polaris\'ees,
il n'est pas clair dans quelle g\'en\'eralit\'e attendre des r\'esultats positifs.

Dans ce texte, on \'etudie le cas particulier des intersections compl\`etes lisses
par des m\'ethodes de th\'eorie g\'eom\'etrique des invariants.
On construit ainsi de nombreux exemples d'espaces de modules quasi-projectifs de vari\'et\'es de Fano.

\subsection{\'Enonc\'e des principaux r\'esultats}

Soient $N\geq 2$, $1\leq c\leq N-1$ et $2\leq d_1\leq \ldots\leq d_c$ des entiers. Une intersection compl\`ete sur un corps $k$
est un sous-sch\'ema de codimension $c$ de $\mathbb{P}^N_k$ d\'efini par $c$ \'equations homog\`enes de degr\'es $d_1,\dots, d_c$.

Soit $H$ l'ouvert du sch\'ema de Hilbert de $\mathbb{P}^N_{\mathbb{Z}}$ param\'etrant les intersections compl\`etes lisses
(voir \cite{Sernesi} 4.6.1). Si on n'a pas $c=1$ et $d_1=2$, l'action par changement de coordonn\'ees de
$PGL_{N+1}$ sur $H$ est propre (\cite{Oolsep} Th\'eor\`eme 1.7),
et le th\'eor\`eme de Keel et Mori \cite{KM} montre l'existence d'un
quotient g\'eom\'etrique $M$ de $H$ par $PGL_{N+1}$, unique par \cite{Kollarquo} Corollary 2.15. C'est un espace
alg\'ebrique s\'epar\'e de type fini sur $\Spec(\mathbb{Z})$ :
l'espace de modules (grossier) des intersections compl\`etes lisses. 

L'espace alg\'ebrique $M$ est-il un sch\'ema ? Un sch\'ema quasi-projectif ? Un sch\'ema affine ?
Le r\'esultat principal de ce texte est le suivant :

\begin{thm}\label{edm}
Soit $M$ l'espace de modules des intersections compl\`etes lisses.
\begin{enumerate}[(i)]
 \item Si $d_1=\dots=d_c$ et si l'on n'a pas $c=1$ et $d_1=2$, $M$ est un sch\'ema affine.
\item Si $c\geq 2$, $d_1<d_2=\dots=d_c$ et $d_2(N-c+2)>d_1((c-1)(d_2-d_1)+1)$, $M$ est un sch\'ema quasi-projectif.
\end{enumerate}
\end{thm} 

  En caract\'eristique nulle, la quasi-projectivit\'e d'un espace de modules de vari\'et\'es lisses dont le fibr\'e canonique est ample 
est connue par les travaux de Viehweg \cite{Viehweg}. 
Or si $c\geq 2$ et $d_1<d_2=\dots=d_c$, et que le fibr\'e canonique
des intersections compl\`etes consid\'er\'ees n'est pas ample, les hypoth\`eses
du th\'eor\`eme \ref{edm} (ii) sont v\'erifi\'ees. En effet, on a $N+1\geq (c-1)d_2+d_1\geq(c-1)(d_2-d_1)+(c-1)+d_1$. 
Ainsi, $d_2(N-c+2)\geq d_2(c-1)(d_2-d_1)+d_2d_1>d_1((c-1)(d_2-d_1)+1)$.
Les r\'esultats de Viehweg et le th\'eor\`eme \ref{edm} impliquent donc :

\begin{cor}
En caract\'eristique nulle, si $d_1\leq d_2=\dots=d_c$ et si l'on n'a pas $c=1$ et $d_1=2$, $M$ est un sch\'ema quasi-projectif.
\end{cor}

\subsection{Le cas $d_1=\dots=d_c$}\label{edmaff}
La preuve du th\'eor\`eme \ref{edm} (ii) occupe la majeure partie de ce texte.
En revanche, le th\'eor\`eme \ref{edm} (i) est facile :
le cas des hypersurfaces ($c=1$)
est d\^u \`a Mumford (\cite{GIT} Prop. 4.2), et la preuve se g\'en\'eralise facilement.

\begin{proof}[$\mathbf{Preuve \text{ }du \text{ }th\acute{e}or\grave{e}me\text{ }\ref{edm} (i)}$]~
  Soit $\bar{H}$ la grassmanienne (relative sur $\Spec(\mathbb{Z})$)
des sous-espaces vectoriels de dimension $c$ de $H^0(\mathbb{P}^N,\mathcal{O}(d_1))$.
Le sch\'ema de Hilbert $H$ s'identifie \`a un ouvert de $\bar{H}$ :
le compl\'ementaire du diviseur discriminant. Comme la grassmanienne est lisse de
groupe de Picard engendr\'e par le fibr\'e de Pl\"ucker, tout diviseur effectif non trivial sur celle-ci est ample. Ainsi, le discriminant
est ample, et son compl\'ementaire $H$ est affine. On pose $H=\Spec(A)$.

  Par \cite{Oolsep} Th\'eor\`eme 1.7, $PGL_{N+1}$ agit proprement sur $H$. Comme de plus $PGL_{N+1}$ est r\'eductif, on peut appliquer un
th\'eor\`eme de Seshadri (\cite{Seshadri} Theorem 3, \cite{Kollarquo} Theorem 7.3) pour montrer que
le quotient g\'eom\'etrique de $H$ par $PGL_{N+1}$ est $M=\Spec(A^{PGL_{N+1}})$, et est donc affine.
\end{proof}

\subsection{Plan du texte}\label{plan}
L'argument de Mumford d\'ecrit ci-dessus fonctionne car $H$ admet une compactification tr\`es simple.
Quand $d_1<d_2=\dots=d_c$, $H$ a encore une compactification explicite $\bar{H}$ :
un fibr\'e en grassmaniennes sur un espace projectif. 
Les paragraphes \ref{constructionsqp} et \ref{parample} sont consacr\'es
\`a la construction et \`a l'\'etude de cette compactification.
Le r\'esultat principal est le th\'eor\`eme \ref{ample} qui calcule son c\^one ample.

On pourrait alors esp\'erer que l'argument de Mumford fonctionne encore :
il faudrait que le diviseur discriminant soit ample sur $\bar{H}$. Malheureusement, ce n'est
jamais le cas si $c=2$ (voir la remarque \ref{echec}). On doit donc appliquer la th\'eorie
g\'eom\'etrique des invariants de mani\`ere moins na\"ive : on fixe un fibr\'e
ample sur $\bar{H}$ et on calcule \`a l'aide du crit\`ere de Hilbert-Mumford quand toutes les intersections compl\`etes lisses sont stables.
C'est l'objet du paragraphe \ref{preuveedm}. 

La preuve de l'in\'egalit\'e qui permet de v\'erifier le crit\`ere de Hilbert-Mumford est report\'ee \`a la troisi\`eme partie : c'est
le th\'eor\`eme \ref{alphadeg}.
Celui-ci est \'enonc\'e et d\'emontr\'e sans hypoth\`eses restrictives sur les degr\'es des intersections compl\`etes.

On peut maintenant expliquer le r\^ole des hypoth\`eses du th\'eor\`eme \ref{edm}. Si l'on n'a pas $d_1\leq d_2=\dots=d_c$,
je ne connais pas de compactification explicite de $H$ analogue \`a celles \'evoqu\'ees ci-dessus.
Si $d_1<d_2=\dots=d_c$, mais qu'on n'a pas $d_2(N-c+2)>d_1((c-1)(d_2-d_1)+1)$, aucun fibr\'e en droites ample sur $\bar{H}$ ne rend toutes
les intersections compl\`etes lisses stables (proposition \ref{stablisse}),
et on ne peut pas appliquer la th\'eorie g\'eom\'etrique des invariants sur $\bar{H}$.

Pour montrer la quasi-projectivit\'e de $M$
pour d'autres valeurs des degr\'es \`a l'aide de th\'eorie g\'eom\'etrique des invariants, il faut donc consid\'erer
une autre compactification de $H$. On peut choisir
(suivant Mumford \cite{Mumfordstab})
le sch\'ema de Hilbert de $\mathbb{P}^N$. Cette possibilit\'e est discut\'ee dans la quatri\`eme partie.
On y explique en particulier
pourquoi l'in\'egalit\'e \ref{alphadeg} est plus faible que celle
qui serait n\'ecessaire \`a la preuve de la Hilbert-stabilit\'e des intersections compl\`etes lisses.

\subsection{Liens avec d'autres travaux}

La th\'eorie g\'eom\'etrique des invariants d'hypersurfaces ou d'intersections comp\-l\`etes
dans $\mathbb{P}^N$ a \'et\'e \'etudi\'ee dans de nombreux cas particuliers :
surfaces quartiques \cite{Shah4}, solides cubiques \cite{Allcock33}, cubiques dans $\mathbb{P}^5$ \cite{Laza34}, pinceaux de
quadriques dans $\mathbb{P}^4$ \cite{AM224}, intersections d'une quadrique et d'une cubique dans $\mathbb{P}^3$
\cite{Lazaetcie}, \cite{Lazaetcie2}, ...

Chacun de ces travaux \'etudie un espace de modules pr\'ecis et m\`ene une analyse compl\`ete : le lieu semi-stable est calcul\'e
et on obtient une compactification de l'espace de modules.
Dans ce texte, on obtient des r\'esultats pour beaucoup de valeurs des degr\'es, mais le r\'esultat est moins fort :
on se contente de montrer que les intersections compl\`etes lisses sont stables.

Signalons particuli\`erement \cite{Lazaetcie} et \cite{Lazaetcie2} o\`u est men\'ee une \'etude tr\`es pr\'ecise du
cas particulier $N=3$, $c=2$, $d_1=2$ et $d_2=3$, en utilisant la compactification $\bar{H}$ \'etudi\'ee dans la suite
de cet article.

\paragraph{Remerciements.}~ Les suggestions d'un rapporteur anonyme ont permis d'am\'e\-lio\-rer
la pr\'esentation de ce texte de mani\`ere importante.

\section{G\'eom\'etrie du sch\'ema $\bar{H}$}\label{partieqp}
\begin{conventions}\label{notationsgen}
Dans cette partie, on fixe $2\leq c\leq N-1$ et $2\leq d_1<d_2=\ldots= d_c$ des entiers.
Une intersection compl\`ete sur un corps $K$ est toujours de codimension $c$
dans $\mathbb{P}^N_K$ et de degr\'es $d_1,\dots,d_c$.

Sauf mention du contraire, les
sch\'emas que nous consid\'ererons seront d\'efinis sur
$\Spec(\mathbb{Z})$. En particulier,
$\mathbb{P}^N=\mathbb{P}^N_{\mathbb{Z}}$.
Quand on manipulera un point g\'eom\'etri\-que, on
notera toujours $K$ le corps alg\'ebriquement clos sur lequel il est d\'efini.

  Si $\mathcal{F}$ est un faisceau localement libre sur un sch\'ema,
le fibr\'e vectoriel g\'eom\'e\-tri\-que associ\'e \`a $\mathcal{F}$ est
celui dont le faisceau des sections est $\mathcal{F}^{\vee}$. Par
$\mathbb{G}(r,\mathcal{F})$, on d\'esignera la grassmannienne des sous-espaces vectoriels de rang
$r$ de ce fibr\'e vectoriel g\'eom\'etrique. Quand $r=1$, on notera aussi ce sch\'ema $\mathbb{P}(\mathcal{F})$.
\end{conventions}

\subsection{Constructions}\label{constructionsqp}
On construit tout d'abord les sch\'emas $H$ et $\bar{H}$, les familles de sous-sch\'emas de $\mathbb{P}^N$ qu'ils param\`etrent,
ainsi que divers faisceaux localement libres sur ces espaces.
On utilisera notamment les notations du diagramme ci-dessous.

$$
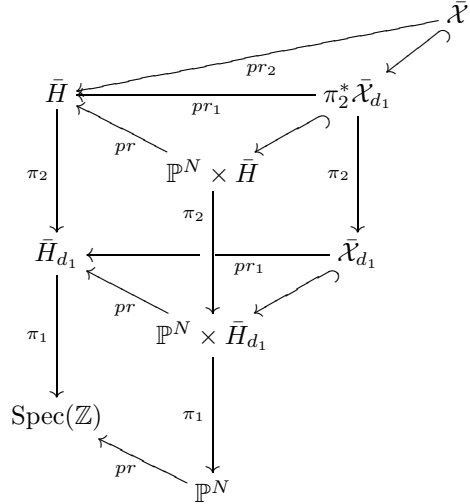
\xymatrix @C=5mm @R=5mm{
&&&\bar{\mathcal{X}}\ar[dlll]^{pr_2}\ar@{^{(}->}[dl]     \\
\bar{H}\ar[dd]_{\pi_2}&&\pi_2^*\bar{\mathcal{X}}_{d_1}\ar[ll]^{pr_1}\ar@{^{(}->}[dl]\ar[dd]_{\pi_2}&                  \\
& \mathbb{P}^N\times\bar{H}\ar[ul]^{pr}\ar[dd]_<<<<{\pi_2}&&  \\
\bar{H}_{d_1}\ar[dd]_{\pi_1}&&\bar{\mathcal{X}}_{d_1}\ar[ll]|\hole^<<<<<<<<<<<{pr_1}\ar@{^{(}->}[dl]&                  \\
& \mathbb{P}^N\times\bar{H}_{d_1}\ar[dd]_{\pi_1}\ar[ul]^{pr}&&  \\
\Spec(\mathbb{Z})&&&\\
&   \mathbb{P}^N   \ar[ul]^{pr}    &&
 }$$

\paragraph{Hypersurfaces.}~
Soit $d\geq 1$. Notons $pr:\mathbb{P}^N\to \Spec(\mathbb{Z})$ le morphisme structurel.
Le faisceau $pr_*\mathcal{O}_{\mathbb{P}^N}(d)$ sur $\Spec(\mathbb{Z})$ est
localement libre, et ses fibres g\'eom\'etriques s'identifient \`a $H^0(\mathbb{P}^N_K,\mathcal{O}(d))$.
 On note
$\bar{H}_{d}=\mathbb{P}((pr_*\mathcal{O}_{\mathbb{P}^N}(d))^{\vee})$ et
$\pi_1:\bar{H}\rightarrow\Spec(\mathbb{Z})$ la projection.
Un point g\'eom\'etrique de
$\bar{H}_d$ est une droite vectorielle $\langle F\rangle$ de
$H^0(\mathbb{P}^N_K,\mathcal{O}(d))$.

\vspace{1em}

On note encore $pr:\mathbb{P}^N\times\bar{H}_d\to\bar{H}_d$ et
$\pi_1:\mathbb{P}^N\times\bar{H}_d\to\mathbb{P}^N$ les changements de base.
La construction de $\bar{H}_d$ fournit une injection du
fibr\'e en droites tautologique
$\mathcal{O}_{\bar{H}_d}(-1)\to\pi_1^*pr_*\mathcal{O}_{\mathbb{P}^N}(d)$. Par changement de base
par le morphisme plat $\pi_1$, cette injection se r\'e\'ecrit
$\mathcal{O}_{\bar{H}_d}(-1)\to
pr_*\mathcal{O}_{\mathbb{P}^N\times\bar{H}_d}(d;0)$. Tirant en
arri\`ere sur $\mathbb{P}^N\times\bar{H}_d$, et utilisant
l'adjonction, on obtient un morphisme de fibr\'es en droites
$\mathcal{O}_{\mathbb{P}^N\times\bar{H}_d}(0;-1)\rightarrow\mathcal{O}_{\mathbb{P}^N\times\bar{H}_d}(d;0)$.
Le lieu o\`u ce morphisme est nul est un diviseur de Cartier $\bar{\mathcal{X}}_d$ sur
$\mathbb{P}^N\times\bar{H}_d$. Par construction, la fibre
en $\langle F\rangle$ de $pr_1:\bar{\mathcal{X}}_d\rightarrow\bar{H}_d$ est
le sous-sch\'ema $\{F=0\}$ de $\mathbb{P}^N_K$.
L'\'equation de $\bar{\mathcal{X}}_{d}$ fournit sur
$\mathbb{P}^N\times\bar{H}_{d}$ la suite exacte courte suivante
:
$$0\rightarrow\mathcal{O}_{\mathbb{P}^N\times\bar{H}_{d}}(-d;-1)\rightarrow
\mathcal{O}_{\mathbb{P}^N\times\bar{H}_{d}}\rightarrow\mathcal{O}_{\bar{\mathcal{X}}_{d}}\rightarrow
0.$$ Tensorisons par
$\mathcal{O}_{\mathbb{P}^N\times\bar{H}_{d}}(l;0)$, et
appliquons $pr_*$ en remarquant par calcul du $H^1$ des fibres que
$R^1pr_*\mathcal{O}_{\mathbb{P}^N\times\bar{H}_{d}}(-d;-1)=0$.
Utilisons la formule de projection et le changement de base par le
morphisme plat $\pi_1$ pour obtenir sur $\bar{H}_{d}$ la suite
exacte courte de faisceaux suivante :
\begin{equation}
0\rightarrow
\pi_1^*pr_*\mathcal{O}_{\mathbb{P}^N}(l-d)\otimes\mathcal{O}_{\bar{H}_{d}}(-1)\rightarrow
\pi_1^*pr_*\mathcal{O}_{\mathbb{P}^N}(l)\rightarrow
pr_{1*}\mathcal{O}_{\bar{\mathcal{X}}_{d}}(l)\rightarrow 0.
\label{faisceau1}
\end{equation}
Par exactitude \`a droite du produit tensoriel, on voit que la fibre
g\'eom\'etrique $(pr_{1*}\mathcal{O}_{\bar{\mathcal{X}}_{d}}(l))_{\langle F\rangle}$
est $H^0(\mathbb{P}^N_K,\mathcal{O}(l))/\left\langle
F\right\rangle$, o\`u l'on a not\'e $\langle F\rangle=H^0(\mathbb{P}^N_K,\mathcal{O}(l-d))\cdot F$. Ainsi,
$pr_{1*}\mathcal{O}_{\bar{\mathcal{X}}_{d}}(l)$ est localement libre
par constance de la dimension de ses fibres, et la fibre
g\'eom\'etrique en $\langle F\rangle$ de la suite exacte courte de faisceaux localement libres (\ref{faisceau1}) est :
\begin{equation}
0\rightarrow\left\langle F\right\rangle\rightarrow
H^0(\mathbb{P}^N_K,\mathcal{O}(l)) \rightarrow
H^0(\mathbb{P}^N_K,\mathcal{O}(l))/\left\langle
F\right\rangle\rightarrow 0.
\label{fibfaisceau1}
\end{equation}

\paragraph{Intersections compl\`etes.}

On a vu ci-dessus que
$pr_{1*}\mathcal{O}_{\bar{\mathcal{X}}_{d_1}}(d_2)$ est un faisceau localement libre sur $\bar{H}_{d_1}$. On notera $\bar{H}
=
\mathbb{G}_{\bar{H}_{d_1}}(c-1,(pr_{1*}\mathcal{O}_{\bar{\mathcal{X}}_{d_1}}(d_2)^{\vee}))$
et $\pi_2:\bar{H}\rightarrow\bar{H}_{d_1}$ la
projection. Par (\ref{fibfaisceau1}), les points g\'eom\'etriques de
$\bar{H}$ sont en bijection avec la donn\'ee d'une droite $\langle F_1\rangle$
 de $H^0(\mathbb{P}^N_K,\mathcal{O}(d_1))$ et d'un sous-espace vectoriel de dimension
$c-1$ de $H^0(\mathbb{P}^N_K,\mathcal{O}(d_2))/\langle F_1\rangle$.
Si $F_2,\dots, F_c\in H^0(\mathbb{P}^N_K,\mathcal{O}(d_2))$ engendrent ce sous-espace vectoriel, on notera
$[F_1,F_2,\dots, F_c]$ ce point g\'eom\'etrique de
$\bar{H}$.
La description de $\bar{H}$ comme grassmannienne relative
sur un espace projectif montre que son groupe de Picard est de rang
$2$, engendr\'e par $\mathcal{O}(1,0)=\pi_2^*\mathcal{O}(1)$ et par
le fibr\'e de Pl\"ucker relatif $\mathcal{O}(0,1)$.

\vspace{1em}

La construction de $\bar{H}$ fournit une injection du
fibr\'e tautologique $\mathcal{F}\to\pi_2^*pr_{1*}\mathcal{O}_{\bar{\mathcal{X}}_{d_1}}(d_2)$. Par changement
de base par le morphisme plat $\pi_2$, cette injection se
r\'e\'ecrit $\mathcal{F}\rightarrow
pr_{1*}\mathcal{O}_{\pi_2^*\bar{\mathcal{X}}_{d_1}}(d_2;0,0)$. Tirant en
arri\`ere sur $\pi_2^*\bar{\mathcal{X}}_{d_1}$, et utilisant l'adjonction,
on obtient un morphisme de fibr\'es vectoriels
$pr^*_{1}\mathcal{F}\rightarrow\mathcal{O}_{\pi_2^*\bar{\mathcal{X}}_{d_1}}(d_2;0,0)$.
Le lieu des z\'eros de ce morphisme est un sous-sch\'ema
$\bar{\mathcal{X}}$ de $\pi_2^*\bar{\mathcal{X}}_{d_1}$. Par
construction, la fibre en $[F_1,F_2,\dots,F_c]$ de la projection
$pr_2:\bar{\mathcal{X}}\rightarrow\bar{H}$ est le
sous-sch\'ema $\{F_1=F_2=\dots=F_c=0\}$ de $\mathbb{P}^N_K$.

Notons $H$ l'ouvert de
$\bar{H}$ constitu\'e des points g\'eom\'etriques $[F_1,F_2,\dots,F_c]$ tels que $\{F_1=F_2=\dots=F_c=0\}$ 
soit lisse de codimension $c$ dans $\mathbb{P}^N_K$. On note $\mathcal{X}\rightarrow H$
la restriction de $\bar{\mathcal{X}}\rightarrow\bar{H}$ \`a $H$. 
On montre ais\'ement que $\mathcal{X}\rightarrow H$
s'identifie au sch\'ema de Hilbert des intersections compl\`etes lisses et \`a sa famille universelle.

\vspace{1em}

Par construction de $\bar{H}$, on dispose d'une suite exacte courte de faisceaux localement libres sur $\bar{H}$ :
\begin{equation}\label{faisceau2}
 0\to\mathcal{F}\to  \pi_2^*pr_{1*}\mathcal{O}_{\bar{\mathcal{X}}_{d_1}}(d_2)\to \mathcal{Q}\to0,
\end{equation}
dont la fibre g\'eom\'etrique en
$[F_1,\dots, F_c]$ s'identifie \`a  :
\begin{equation}\label{fibfaisceau2}
  \scriptstyle{0\to\langle F_2,\dots,F_c\rangle\to H^0(\mathbb{P}^N_K,\mathcal{O}(d_2))/\langle F_1\rangle
\to H^0(\mathbb{P}^N_K,\mathcal{O}(d_2))/\langle F_1,\dots,F_c\rangle\to0.}
\end{equation}

Par ailleurs, par (\ref{faisceau1}) pour $d=d_1$ et $l=d_2$,
on dispose d'une suite exacte courtes de faisceaux localement libres sur $\bar{H}$ :
\begin{equation}\label{faisceau3}
 \scriptstyle{0\to\pi_2^*\pi_1^*pr_*\mathcal{O}_{\mathbb{P}^N}(d_2-d_1)\otimes\mathcal{O}(-1,0)}\to\scriptstyle{
\pi_2^*\pi_1^*pr_{*}\mathcal{O}_{\mathbb{P}^N}(d_2)}\to\scriptstyle{\pi_2^*pr_{1*}\mathcal{O}_{\bar{\mathcal{X}}_{d_1}}(d_2)}\to 0.
\end{equation} 
Par (\ref{fibfaisceau1}), la fibre g\'eom\'etrique de (\ref{faisceau3}) en
$[F_1,\dots, F_c]$ s'identifie \`a  :
\begin{equation}\label{fibfaisceau3}
 0\to\langle F_1\rangle\to H^0(\mathbb{P}^N_K,\mathcal{O}(d_2))
\to H^0(\mathbb{P}^N_K,\mathcal{O}(d_2))/\langle F_1\rangle\to0.
\end{equation}

Notons $\mathcal{E}$ le faisceau localement libre noyau de la
compos\'ee des surjections
$\pi_2^*\pi_1^*pr_{*}\mathcal{O}_{\mathbb{P}^N}(d_2)\rightarrow
\pi_2^*pr_{1*}\mathcal{O}_{\bar{\mathcal{X}}_{d_1}}(d_2)\rightarrow
\mathcal{Q}$. Les suites exactes (\ref{faisceau2}), (\ref{fibfaisceau2}), (\ref{faisceau3}) et
(\ref{fibfaisceau3}) permettent d'\'ecrire le diagramme
exact (\ref{diagfaisc}) de faisceaux localement libres sur $\bar{H}$ ci-dessous et de calculer sa fibre g\'eom\'etrique
(\ref{diagfibre}) en $[F_1,\dots,F_c]$ :

\begin{equation}
\begin{split}
\xymatrix @C=2mm @R=3mm {
&0\ar[d]&0\ar[d]&
\\
&\scriptstyle{\pi_2^*\pi_1^*pr_{*}\mathcal{O}_{\mathbb{P}^N}(d_2-d_1)\otimes\mathcal{O}(-1,0)}\ar@{=}[r]\ar[d]
&\scriptstyle{\pi_2^*\pi_1^*pr_{*}\mathcal{O}_{\mathbb{P}^N}(d_2-d_1)\otimes\mathcal{O}(-1,0)}\ar[d]&        \\
0\ar[r]&\mathcal{E}\ar[r]\ar[d]&\scriptstyle{\pi_2^*\pi_1^*pr_{*}\mathcal{O}_{\mathbb{P}^N}(d_2)}\ar[r]\ar[d]
&\mathcal{Q} \ar[r]\ar@{=}[d] & 0                 \\
0\ar[r]&\mathcal{F}\ar[r]\ar[d]&
\scriptstyle{\pi_2^*pr_{1*}\mathcal{O}_{\bar{\mathcal{X}}_{d_1}}(d_2)}\ar[r]\ar[d]
&\mathcal{Q} \ar[r]&0             \\
&0&0 }
\label{diagfaisc}
\end{split}
\end{equation}
%
%Par (\ref{fibfaisceau2}) et (\ref{fibfaisceau3}), la fibre g\'eom\'etrique en $[F_1,\dots,F_c]$ de ce diagramme s'identifie \`a :
%
%
\begin{equation}
\begin{split}
\xymatrix @C=2mm @R=3mm {
&0\ar[d]&0\ar[d]&
\\
&\scriptstyle{\langle F_1\rangle}\ar@{=}[r]\ar[d]
&\scriptstyle{\langle F_1\rangle}\ar[d]&        \\
0\ar[r]&\scriptstyle{\langle
F_1,\dots, F_c\rangle}\ar[r]\ar[d]&\scriptstyle{H^0(\mathbb{P}^N_K,\mathcal{O}(d_2))}\ar[r]\ar[d]
&\scriptstyle{H^0(\mathbb{P}^N_K,\mathcal{O}(d_2))/\langle
F_1,\dots, F_c\rangle }\ar[r]\ar@{=}[d] & 0                 \\
0\ar[r]&\scriptstyle{\langle
F_1,\dots, F_c\rangle/\langle F_1\rangle}\ar[r]\ar[d]&
\scriptstyle{H^0(\mathbb{P}^N_K,\mathcal{O}(d_2))/\langle F_1\rangle}\ar[r]\ar[d]
&\scriptstyle{H^0(\mathbb{P}^N_K,\mathcal{O}(d_2))/\langle
F_1,\dots, F_c\rangle} \ar[r]&0             \\
&0&0 }
\label{diagfibre}
\end{split}
\end{equation}

\paragraph{Action de $SL_{N+1}$.}~
L'action de $SL_{N+1}$ sur $\mathbb{A}^{N+1}$ induit des actions de
$SL_{N+1}$ par changement de coordonn\'ees sur tous les espaces et
faisceaux d\'ecrits ci-dessus.
En particulier, $SL_{N+1}$ agit sur le faisceau localement libre
$pr_*\mathcal{O}_{\mathbb{P}^N}(d_1)$,
induisant une lin\'earisation de
$\mathcal{O}(1)$ sur $\bar{H}_{d_1}$. Par fonctorialit\'e, on en
d\'eduit une lin\'earisation de $\mathcal{O}(1,0)$ sur
$\bar{H}$.
De m\^eme, $SL_{N+1}$ agit sur le faisceau localement libre
$pr_{1*}\mathcal{O}_{\mathcal{X}_{d_1}}(d_2)$, donc sur 
$\bigwedge^{c-1}(pr_{1*}\mathcal{O}_{\mathcal{X}_{d_1}}(d_2))$,
induisant une lin\'earisation du fibr\'e de Pl\"ucker relatif $\mathcal{O}(0,1)$ sur $\bar{H}$.
Par combinaisons lin\'eaires, on construit alors une lin\'earisation
naturelle de tous les fibr\'es en droites $\mathcal{O}(l_1,l_2)$ sur
$\bar{H}$. Ces lin\'earisations sont uniques par
\cite{GIT}, Prop. 1.4.

\subsection{Fibr\'es amples sur $\bar{H}$}\label{parample}
Avant de pouvoir prouver le th\'eor\`eme \ref{ample} qui d\'ecrit les fibr\'es en droites amples sur $\bar{H}$,
on a besoin r\'esultats pr\'eliminaires sur la g\'eom\'etrie de $\bar{H}$.

\paragraph{Lien entre $\bar{H}$ et $\bar{H}_{d_1}\times \bar{H}_{d_2}^{c-1}$.}
L'injection de faisceaux localement libres $\mathcal{E}\rightarrow
\pi_2^*\pi_1^*pr_{*}\mathcal{O}_{\mathbb{P}^N}(d_2)$ dans le diagramme
(\ref{diagfaisc}) induit une immersion ferm\'ee
$\mathbb{P}\mathcal{E}^{\vee}\hookrightarrow\bar{H}\times\bar{H}_{d_2}$ entre fibr\'es projectifs sur $\bar{H}$.
En prenant le produit fibr\'e au-dessus de $\bar{H}$ de $c-1$ copies de ces fibr\'es projectifs, on obtient une immersion ferm\'ee
$i:\Sigma\hookrightarrow\bar{H}\times\bar{H}_{d_2}^{c-1}$.

Remarquons que, par le diagramme (\ref{diagfibre}), les points
g\'eom\'etriques de $\Sigma$ sont les
$([F_1,\dots,F_c],\langle G_2\rangle,\dots,\langle G_c\rangle)\in(\bar{H}\times\bar{H}_{d_2}^{c-1})(K)$
tels que $G_i\in\left\langle F_1,\dots,F_c \right\rangle$ pour $2\leq i\leq c$. Le diagramme
ci-dessous r\'esume les notations que nous utiliserons.

$$\xymatrix{
&\bar{H}\times\bar{H}_{d_2}^{c-1}\ar[ddl]_{\pi_2\times id}\ar[ddr]^{p_1}&        \\
&\Sigma\ar^{i}[u]\ar[dl]^{e}\ar[dr]_{q}&                   \\
\bar{H}_{d_1}\times\bar{H}_{d_2}^{c-1}\ar[dr]_{p_1}&& \bar{H}\ar[dl]^{\pi_2}     \\
&\bar{H}_{d_1}&}$$

Notre objectif est de comparer les espaces $\bar{H}$
et $\bar{H}_{d_1}\times \bar{H}_{d_2}^{c-1}$ via
$\Sigma$. On commence par \'etudier l'application
$e:=(\pi_2\times id)\circ i$.
La description des points g\'eom\'etriques de $\Sigma$ et de $\bar{H}_{d_1}\times \bar{H}_{d_2}^{c-1}$
montre que le ferm\'e de $\bar{H}_{d_1}\times \bar{H}_{d_2}^{c-1}$
o\`u $e$ a des fibres de dimension $>0$ a pour points g\'eom\'etriques les
$(\langle F_1\rangle,\langle G_2\rangle,\dots,\langle G_c\rangle)$ tels que 
$\langle F_1\rangle\cap\langle G_2,\dots,G_c\rangle\neq\{0\}$. On note ce ferm\'e $W$, et on le munit de sa structure r\'eduite.
Notons $E$ le ferm\'e $e^{-1}(W)$ de $\Sigma$, et munissons-le de sa structure r\'eduite ($E=e^{-1}(W)$ vaut ensemblistement
mais pas n\'ecessairement sch\'ematiquement).
Les points g\'eom\'etriques de $E$ sont les points g\'eom\'etriques
$([F_1,\dots,F_c],\langle G_2\rangle,\dots,\langle G_c\rangle)$ de $\Sigma$
tels que $\langle F_1,G_2,\dots,G_c\rangle\subsetneq\langle F_1,\dots,F_c\rangle$.

\begin{lemme}\label{EWirr}
Les sch\'emas $E$ et $W$ sont irr\'eductibles.
\end{lemme}

\begin{proof}[$\mathbf{Preuve}$]
 Comme $W=e(E)$, il suffit de montrer que $E$ est irr\'eductible. Pour cela, il suffit de montrer que les fibres g\'eom\'etriques 
de $q|_E:E\to\bar{H}$ le sont.
Soit $[F_1,\dots,F_c]$ un point g\'eom\'etrique de $\bar{H}$. 
La fibre g\'eom\'etrique de $q|_E:E\to\bar{H}$ correspondante est constitu\'ee des $(\langle G_2\rangle,\dots,\langle G_c\rangle)$ n'induisant
pas une base de $\langle F_1,\dots,F_c\rangle/\langle F_1\rangle$. 
Elle est donc ensemblistement d\'efinie par l'annulation d'un d\'eterminant,
et irr\'eductible par irr\'eductibilit\'e du d\'eterminant.
\end{proof}

\begin{lemme}\label{birat}
Le morphisme $e|_{\Sigma\setminus E}:\Sigma\setminus E\to (\bar{H}_{d_1}\times \bar{H}_{d_2}^{c-1})\setminus W$ est un isomorphisme.
\end{lemme}

\begin{proof}[$\mathbf{Preuve}$]
Le morphisme $e=(\pi_2\times id)\circ i$ est propre comme compos\'ee, donc, par changement de base, 
$e|_{\Sigma\setminus E}$ est propre. De plus, la description des points g\'eom\'etriques de
$\Sigma$ et $\bar{H}_{d_1}\times \bar{H}_{d_2}^{c-1}$
montre que $e|_{\Sigma\setminus E}$ r\'ealise une bijection entre points g\'eom\'etriques. Ainsi, $e|_{\Sigma\setminus E}$ est propre
et quasifini, donc fini. Finalement, par lissit\'e g\'en\'erique ($\bar{H}_{d_1}\times \bar{H}_{d_2}^{c-1}$ est
de caract\'eristique g\'en\'erique $0$), $e^{-1}(\langle F_1\rangle,\langle G_2\rangle,\dots,\langle G_c\rangle)$ est un point r\'eduit pour 
$(\langle F_1\rangle,\langle G_2\rangle,\dots,\langle G_c\rangle)$ g\'e\-n\'e\-rique, de sorte que
$e|_{\Sigma\setminus E}$ est birationnel. Comme $\bar{H}_{d_1}\times \bar{H}_{d_2}^{c-1}$ est
r\'egulier donc normal, par le Main Theorem de Zariski,
$e|_{\Sigma\setminus E}$ est un isomorphisme.
\end{proof}

\paragraph{Des \'equations pour $E$ et $W$.}~
\begin{prop}\label{E}
Le sous-sch\'ema $E$ est un diviseur de Cartier dans $\Sigma$ et $\mathcal{O}(E)=i^*\mathcal{O}(0,-1,1,\dots,1)$.
\end{prop}

\begin{proof}[$\mathbf{Preuve}$]
La construction de $\Sigma$ comme produit de fibr\'es projectifs sur $\bar{H}$ fournit $c-1$ sous-faisceaux tautologiques
$\mathcal{L}_1,\dots,\mathcal{L}_{c-1}$ de $q^*\mathcal{E}$.
 On a
donc un morphisme $\bigoplus_{k=1}^{c-1}\mathcal{L}_k\rightarrow
q^*\mathcal{E}$ dont la fibre g\'eom\'etrique en $([F_1,\dots,F_c],\langle G_2\rangle,\dots,\langle G_c\rangle)$ est 
$\langle G_2\rangle\oplus\dots\oplus\langle G_c \rangle \rightarrow\langle F_1,\dots,F_c\rangle$.
Remarquons que, par compatibilit\'e entre les faisceaux tautologiques des fibr\'es projectifs $\mathbb{P}\mathcal{E}^{\vee}$
et $\bar{H}\times\bar{H}_{d_2}$
sur $\bar{H}$, on a $\mathcal{L}_1=i^*\mathcal{O}(0,0,-1,0\dots,0),\dots,\mathcal{L}_{c-1}=i^*\mathcal{O}(0,0,\dots,0,-1)$.

D'autre part, en tirant en arri\`ere sur $\Sigma$ le
morphisme $\mathcal{E}\rightarrow\mathcal{F}$ du diagramme
(\ref{diagfaisc}), on obtient un morphisme
$q^*\mathcal{E}\rightarrow q^*\mathcal{F}$ dont la fibre g\'eom\'etrique en
$([F_1,\dots,F_c],\langle G_2\rangle,\dots,\langle G_c\rangle)$ est $\left\langle F_1,\dots,F_c
\right\rangle\rightarrow\left\langle F_1,\dots,F_c
\right\rangle/\left\langle F_1 \right\rangle$ par le diagramme (\ref{diagfibre}). Notons
$\beta:\bigoplus_{k=1}^{c-1}\mathcal{L}_k\rightarrow q^*\mathcal{F}$ la
compos\'ee de ces deux morphismes de faisceaux.

Les fibres
$\beta_{([F_1,\dots,F_c],\langle G_2\rangle,\dots,\langle G_c\rangle)}:
\langle G_2\rangle\oplus\dots\oplus\langle G_c \rangle \rightarrow\langle F_1,\dots,F_c
\rangle/\langle F_1 \rangle$ de $\beta$ 
sont des isomorphismes exactement si $([F_1,\dots,F_c],\langle G_2\rangle,\dots,\langle G_c\rangle)\notin E(K)$.
On en d\'e\-duit que $\det(\beta)$ est une
injection, et que son conoyau $\mathcal{K}$ a pour support un
sous-sch\'ema ferm\'e de $\Sigma$ dont la r\'eduction
est $E$.

On remarque alors que $\det(\bigoplus_{k=1}^{c-1}\mathcal{L}_k)=i^*\mathcal{O}(0,0,-1,\dots,-1)$
et que, par d\'efinition du fibr\'e de Pl\"ucker, 
$\det(q^*\mathcal{F})=i^*\mathcal{O}(0,-1,0,\dots,0)$.
Tensorisant par $i^*\mathcal{O}(0,1,0,\dots,0)$, on obtient :
$$0\rightarrow i^*\mathcal{O}(0,1,-1,\dots,-1)\rightarrow
\mathcal{O}_{\Sigma}\rightarrow\mathcal{K}\otimes
i^*\mathcal{O}(0,1,0,\dots,0)\rightarrow 0.$$
 Le
fibr\'e en droites $i^*\mathcal{O}(0,1,-1,\dots,-1)$ s'identifie ainsi au
faisceau d'id\'eaux d'un diviseur de Cartier $D$ de
$\Sigma$ qui co\"incide ensemblistement avec $E$.

Le sous-sch\'ema $E$ est donc le diviseur de Cartier r\'eduit associ\'e \`a $D$ sur le sch\'ema r\'egulier $\Sigma$.
Comme, par le lemme \ref{EWirr}, $E$ est irr\'eductible, il existe $k\geq1$ tel que
$i^*\mathcal{O}(0,1,-1,\dots,-1)=\mathcal{O}(-kE)$. Or la description de $\Sigma$ comme produit de fibr\'es projectifs montre que
$i^*\mathcal{O}(0,1,-1,\dots,-1)$ n'est pas divisible dans
$\Pic(\Sigma)$. On a donc n\'ecessairement $k=1$, et
$\mathcal{O}(E)=i^*\mathcal{O}(0,-1,1,\dots,1)$.
\end{proof}

Les calculs que nous m\`enerons par la suite n\'ecessitent
d'avoir des \'equations au moins ensemblistes pour $W$.
C'est l'objet de la proposition \ref{eqZ}. 

\begin{prop}\label{eqZ}

Soit $(\langle F_1\rangle,\langle G_2\rangle,\dots,\langle G_c\rangle)$ un point g\'eom\'etrique de
$\bar{H}_{d_1}\times\bar{H}_{d_2}^{c-1}$ n'appartenant pas \`a
$W$. Alors il existe un diviseur
$D\in|\mathcal{O}_{\bar{H}_{d_1}\times\bar{H}_{d_2}^{c-1}}((c-1)(d_2-d_1)+1,1,\dots,1)|$
contenant $W$ mais pas $(\langle F_1\rangle,\langle G_2\rangle,\dots,\langle G_c\rangle)$.
\end{prop}

\begin{proof}[$\mathbf{Preuve}$]
Soit $(X_0,\dots,X_N)$ un syst\`eme de coordonn\'ees sur
$\mathbb{P}^N$, c'est-\`a-dire une base de
$H^0(\mathbb{P}^N,\mathcal{O}(1))$. On note $\mathfrak{M}_d$
l'ensemble des mon\^omes de degr\'e $d$ en les $X_i$ : c'est une
base de $H^0(\mathbb{P}^N,\mathcal{O}(d))$. On obtient des
coordonn\'ees sur les espaces projectifs $\bar{H}_{d_1}$ et
$\bar{H}_{d_2}$ en consid\'erant les bases duales
$(a_L)_{L\in\mathfrak{M}_{d_1}}$ et $(b_M^{(i)})_{M\in\mathfrak{M}_{d_2}}$
de
$H^0(\mathbb{P}^N,\mathcal{O}(d_1))^{\vee}=H^0(\bar{H}_{d_1},\mathcal{O}(1))$
et
$H^0(\mathbb{P}^N,\mathcal{O}(d_2))^{\vee}=H^0(\bar{H}_{d_2},\mathcal{O}(1))$,
o\`u l'exposant $i$ ($2\leq i\leq c$) permet de distinguer les coordonn\'ees sur les $c-1$ copies de $\bar{H}_{d_2}$.
On choisit notre syst\`eme de coordonn\'ees de sorte que $F_1$ ait
un coefficient non nul en $X_0^{d_1}$, qu'on peut alors supposer
\'egal \`a $1$.

 Soit $2\leq i\leq c$. Consid\'erons l'identit\'e
\begin{equation}
a_{X_0^{d_1}}^{d_2-d_1+1}g^{(i)}=q_{d_2-d_1+1}^{(i)}f+r_{d_2-d_1+1}^{(i)}
\label{diveucl}
\end{equation}
obtenue en substituant la variable $b^{(i)}_M$ \`a la variable $b_M$ dans l'identit\'e fournie par le lemme \ref{division} ci-dessous. 
Substituant alors
les coefficients de $F_1$ dans
les $a_L$ et les coefficients de $G_i$ dans les $b^{(i)}_M$, on obtient
une \'egalit\'e de la forme $G_i=Q_iF_1+R_i$ dans $K[X_0,\dots, X_N]$. Comme
$(\langle F_1\rangle,\langle G_2\rangle,\dots,\langle G_c\rangle)\notin W$, les $R_i$ forment une famille libre. 
On peut donc trouver $(M_j)_{2\leq j\leq c}$ des mon\^omes de $\mathfrak{M}_{d_2}$ 
tels que la matrice $(c-1)\times(c-1)$ dont le coefficient $(i,j)$ est le coefficient de $M_j$ dans $R_i$ soit inversible.

On note $C^{(i)}_j\in\mathbb{Z}[a_L,b^{(i)}_M]_{L\in\mathfrak{M}_{d_1},M\in\mathfrak{M}_{d_2}}$ le coefficient de $M_j$
dans $r_{d_2-d_1+1}^{(i)}$.
Alors $P=\det(C^{(i)}_j)$
est un polyn\^ome homog\`ene de degr\'e
$(c-1)(d_2-d_1+1)$ en les $a_L$ et, pour tout $i\in\{2,\dots,c-1\}$, de degr\'e $1$ en les $b^{(i)}_M$. On voit
$P$ comme une section de $\mathcal{O}_{\bar{H}_{d_1}\times\bar{H}_{d_2}^{c-1}}((c-1)(d_2-d_1+1),1,\dots,1)$. Par choix des $M_j$,
$P$ est non nul en $(\langle F_1\rangle,\langle G_2\rangle,\dots,\langle G_c\rangle)$.

  Montrons que $\{P=0\}$ contient $W$. Comme, par le lemme \ref{EWirr}, $W$ est int\`egre,
il suffit de voir que $\{P=0\}$ contient les points g\'eom\'etriques de
l'ouvert dense de $W$ d\'efini par l'\'equation
$a_{X_0^{d_1}}\neq0$. Soit donc $(\langle F_1'\rangle,\langle G_2'\rangle,\dots,\langle G_c'\rangle)$ un point
g\'eom\'etrique de $W$ tel que le coefficient en $X_0^{d_1}$ de
$F'_1$ vale $1$. Comme $(\langle F_1'\rangle,\langle G_2'\rangle,\dots,\langle G_c'\rangle)\in W$, il existe une \'equation de la forme
$\sum_{i=2}^c\lambda_iG'_i=QF'_1$ avec $Q\in K[X_0,\dots, X_N]$ et $\lambda_i\in K$ non tous nuls.
 Pour $2\leq i\leq c$, en substituant dans l'\'egalit\'e (\ref{diveucl}) les coefficients de $F'_1$
dans les $a_L$ et les coefficients de $G'_i$ dans les $b^{(i)}_M$, on
obtient des \'egalit\'es de la forme $G'_i=Q'_iF'_1+R'_i$. Il vient
$\sum_{i=2}^c\lambda_i R'_i=(Q-\sum_{i=2}^c\lambda_i Q'_i)F_1'$. Comme aucun des mon\^omes des $R'_i$ n'est divisible
par $X_0^{d_1}$ et que le coefficient en $X_0^{d_1}$ de $F_1'$ est
non nul, on a n\'ecessairement $Q-\sum_{i=2}^c\lambda_i Q'_i=0$, donc $\sum_{i=2}^c\lambda_i R'_i=0$. Par
cons\'equent, $P(F'_1,G'_2,\dots, G_c')$ est le d\'eterminant d'une matrice dont les lignes sont li\'ees, et est nul.
Ceci montre que $(\langle F_1'\rangle,\langle G_2'\rangle,\dots,\langle G_c'\rangle)\in \{P=0\}$.
 
  Enfin, remarquons que $P$ est divisible par $a_{X_0^{d_1}}^{c-2}$.
Pour cela, utilisons la derni\`ere partie du lemme \ref{division} : on a une 
identit\'e de la forme $r_{d_2-d_1+1}=a_{X_0^{d_1}}T+b_{X_0^{d_2}}S$.
Par homog\'en\'eit\'e, $S$ ne d\'epend pas des variables $(b_M)_{M\in\mathfrak{M}_{d_2}}$,
de sorte qu'on obtient, pour $2\leq i\leq c$ des identit\'es de la forme
$r^{(i)}_{d_2-d_1+1}=a_{X_0^{d_1}}T^{(i)}+b^{(i)}_{X_0^{d_2}}S$. Ces expressions montrent
que, dans la matrice $(C^{(i)}_j)$, chaque ligne est somme de deux termes :
les premiers divisibles par $a_{X_0^{d_1}}$, les seconds tous proportionnels. 
D\'evelopper le d\'eterminant montre que $P$ est divisible par $a_{X_0^{d_1}}^{c-2}$.

Posons alors $\tilde{P}=P/a_{X_0^{d_1}}^{c-2}$ : c'est une section de
$\mathcal{O}_{\bar{H}_{d_1}\times\bar{H}_{d_2}^{c-1}}((c-1)(d_2-d_1)+1,1,\dots,1)$.
Comme $P$ est non nul en $(\langle F_1\rangle,\langle G_2\rangle,\dots,\langle G_c\rangle)$,
c'est aussi le cas de $\tilde{P}$. Comme $W\subset\{P=0\}$, que $W$ est int\`egre par le lemme \ref{EWirr},
et que $W$ n'est pas inclus dans $\{a_{X_0^{d_1}}=0\}$, $W\subset\{\tilde{P}=0\}$. On a montr\'e que $D=\{\tilde{P}=0\}$ convenait.
\end{proof}

\begin{lemme}\label{division}
 On se place dans l'anneau
$$A=\mathbb{Z}[X_s,a_L,b_M]_{0\leq s\leq N, L\in\mathfrak{M}_{d_1},
M\in\mathfrak{M}_{d_2}}$$
trigradu\'e par le degr\'e total en les
$X_i$, en les $a_L$ et les $b_M$. On consid\`ere les \'el\'ements
$f=\sum_{L\in\mathfrak{M}_{d_1}}a_LL$ et
$g=\sum_{M\in\mathfrak{M}_{d_2}}b_MM$ de $A$. Alors, si $0\leq
j\leq d_2-d_1+1$, il existe $q_j, r_j\in A$ homog\`enes de degr\'es
respectifs $(d_2-d_1,j-1,1)$ et $(d_2,j,1)$, tels qu'aucun mon\^ome
de $r_j$ ne soit divisible par $X_0^{d_2+1-j}$ et que
$$a_{X_0^{d_1}}^jg=q_jf+r_j.$$

De plus, si $j\geq 1$, tout mon\^ome intervenant dans $r_j$ est divisible soit par $a_{X_0^{d_1}}$ soit par
$b_{X_0^{d_2}}$.
\end{lemme}

\begin{proof}[$\mathbf{Preuve}$]
L'existence de $q_j$ et $r_j$ r\'esulte de l'algorithme de division euclidienne.

Plus pr\'ecis\'ement, on raisonne par r\'ecurrence sur $j$. Si
$j=0$, on prend $q_0=0$ et $r_0=g$. Pour passer de l'\'egalit\'e
pour $j$ \`a celle pour $j+1$, on multiplie par $a_{X_0^{d_1}}$, on
regroupe dans $a_{X_0^{d_1}}r_j$ les termes divisibles par
$X_0^{d_2-j}$, et on r\'e\'ecrit ces termes en utilisant
l'identit\'e :
$$a_{X_0^{d_1}}X_0^{d_2-j}=X_0^{d_2-d_1-j}f+X_0^{d_2-d_1-j}(a_{X_0^{d_1}}X_0^{d_1}-f).$$
Cette construction explicite permet facilement de v\'erifier la derni\`ere propri\'et\'e par r\'ecurrence sur $j$. 
\end{proof}

\paragraph{Calcul des fibr\'es amples.}~
On peut \`a pr\'esent montrer :
\begin{thm}\label{ample}
Le fibr\'e $\mathcal{O}(l_1,l_2)$ sur $\bar{H}$ est
ample si et seulement si $l_2>0$ et $\frac{l_1}{l_2}>(c-1)(d_2-d_1)+1$.
\end{thm}

\begin{proof}[$\mathbf{Preuve}$]
Comme $\Spec(\mathbb{Z})$ est affine, $\mathcal{O}(l_1,l_2)$ est ample si et seulement s'il est ample
relativement \`a $\Spec(\mathbb{Z})$. Par \cite{EGA31} 4.7.1, il suffit de tester
l'amplitude de $\mathcal{O}(l_1,l_2)$ sur les fibres
du morphisme structurel, donc sur les fibres g\'eom\'etriques du morphisme structurel.
La proposition est alors cons\'equence de la proposition \ref{nef}
ci-dessous et du crit\`ere de Kleiman : pour une vari\'et\'e
projective sur un corps alg\'ebriquement clos, le c\^one ample est
l'int\'erieur du c\^one nef (\cite{LazarsfeldI}, 1.4.23).
\end{proof}

\begin{prop}\label{nef}
Soit $K$ un corps alg\'ebriquement clos. Alors le fibr\'e
$\mathcal{O}(l_1,l_2)$ sur $\bar{H}\times_{\mathbb{Z}}
K$ est nef si et seulement si $l_2\geq0$ et $l_1\geq
l_2((c-1)(d_2-d_1)+1)$.
\end{prop}

\begin{proof}[$\mathbf{Preuve}$]

Dans toute cette preuve, les vari\'et\'es qu'on manipule sont
d\'efinies sur le corps $K$. Les
extensions des scalaires \`a $K$ seront partout sous-entendues.

\begin{etape1}

La condition est n\'ecessaire.

\end{etape1}

Supposons que $\mathcal{O}(l_1,l_2)$ est nef. On a $l_2\geq0$ car
$\mathcal{O}(l_1,l_2)$ est $\pi_2$-nef. On va montrer la seconde
in\'egalit\'e en calculant le degr\'e de $\mathcal{O}(l_1,l_2)$ sur
une courbe bien choisie.

 Soient $X_0, X_1\in H^0(\mathbb{P}^N,\mathcal{O}(1))$ des \'equations
lin\'eairement ind\'ependantes, $H\in
H^0(\mathbb{P}^N,\mathcal{O}(d_1-1))$ une \'equation non nulle et
$(\lambda^{(i)}_j)_{2\leq i\leq c,1\leq j\leq d_2-d_1}$ des scalaires distincts. Pour $2\leq i\leq c$, on
note $G_i=HX_{i-1}\prod_{j=1}^{d_2-d_1}(X_0+\lambda^{(i)}_jX_1)$. Consid\'erons
$\beta:\mathbb{P}^1\rightarrow\bar{H}_{d_1}$ le pinceau $t\mapsto\langle H(X_0+tX_1)\rangle$. La section constante
$s:\bar{H}_{d_1}\rightarrow\bar{H}_{d_1}\times\bar{H}_{d_2}^{c-1}$
de valeur $(\langle G_2\rangle,\dots,\langle G_c\rangle)$ fournit un morphisme $s\circ
\beta:\mathbb{P}^1\rightarrow\bar{H}_{d_1}\times\bar{H}_{d_2}^{c-1}$.

 Calculons les points de $\mathbb{P}^1$ dont l'image par $s\circ \beta$ est dans $W$.
Soient $t\in\mathbb{P}^1(K)$ et $a_2,\dots, a_c\in K$ non tous nuls. Alors $H(X_0+tX_1)$ divise $\sum_{i=2}^c a_iG_i$
si et seulement si $X_0+tX_1$ divise $\sum_{i=2}^c a_iX_{i-1}\prod_{j=1}^{d_2-d_1}(X_0+\lambda^{(i)}_jX_1)$. On voit ais\'ement que cela ne peut se produire
que si tous les $a_i$ sauf un sont nuls. Si c'est $a_i$ qui est non nul, les valeurs possibles de $t$ sont soit $t=\lambda_j^{(i)}$ pour un
$j\in\{1,\dots,d_2-d_1\}$, soit $t=\infty$
si $i=2$. On a montr\'e qu'exactement $(c-1)(d_2-d_1)+1$ points de $\mathbb{P}^1$ sont envoy\'es dans $W$ par $s\circ \beta$.

 Comme l'image de $s\circ\beta$ n'est pas incluse dans $W$ et que $e$ est birationnel par le lemme \ref{birat}, le crit\`ere valuatif de
propret\'e permet de relever $s\circ\beta$ en un morphisme
$\gamma:\mathbb{P}^1\rightarrow\Sigma$. Remarquons qu'exactement $(c-1)(d_2-d_1)+1$ points
de $\mathbb{P}^1$ sont envoy\'es dans $E$ par $\gamma$. 
Finalement,
en composant par $q$, on obtient un morphisme $q\circ
\gamma:\mathbb{P}^1\rightarrow\bar{H}$.

  On calcule alors les degr\'es des fibr\'es en droites de $\bar{H}$ sur $\mathbb{P}^1$:
\begin{alignat*}{3}
   \mathbb{P}^1\cdot\gamma^*q^*\mathcal{O}(1,0)  &=\mathbb{P}^1\cdot\gamma^*q^*\pi_2^*\mathcal{O}(1) &&= \mathbb{P}^1\cdot\gamma^*e^*p_1^*\mathcal{O}(1)\\
      &= \mathbb{P}^1\cdot\beta^*s^*p_1^*\mathcal{O}(1) &&= \mathbb{P}^1\cdot\beta^*\mathcal{O}(1)\\
  & = 1 &&\text{ car $\beta:\mathbb{P}^1\to\bar{H}_{d_1}$ est une droite.}
\end{alignat*}
\begin{alignat*}{6}
   \mathbb{P}^1\cdot&\gamma^*q^*\mathcal{O}(0,1)=\mathbb{P}^1\cdot\gamma^*e^*\mathcal{O}(0,1,\dots,1)-\mathbb{P}^1\cdot\gamma^*\mathcal{O}(E) &&\\
&\text{\hspace{5em} par la proposition }\ref{E}&&\\
      &\leq \mathbb{P}^1\cdot\beta^*s^*\mathcal{O}(0,1,\dots,1)-(c-1)(d_2-d_1)-1 &&\\
&\text{\hspace{5em} par calcul de }\Card(\gamma^{-1}(E))&&\\
& =-(c-1)(d_2-d_1)-1  &&\\
&\text{\hspace{5em} car }s^*\mathcal{O}(0,1,\dots,1)=\mathcal{O}.
\end{alignat*}
On montre enfin l'in\'egalit\'e voulue comme suit :
\begin{alignat*}{2}
   l_1-l_2((c-1)(d_2-d_1)+1)  &\geq \mathbb{P}^1\cdot\gamma^*q^*\mathcal{O}(l_1,l_2) &&\text{ car }l_2\geq0 \\
      &\geq0 &&\text{ car }\mathcal{O}(l_1,l_2) \text{ est nef. }
\end{alignat*}
\begin{etape2}

La condition est suffisante.

\end{etape2}

Supposons \`a pr\'esent les in\'egalit\'es v\'erifi\'ees, et
montrons que $\mathcal{O}(l_1,l_2)$ est nef. Soit pour cela $C$ une
courbe int\`egre de $\bar{H}$. Notons $\tilde{C}$ sa
normalisation et $\alpha:\tilde{C}\rightarrow\bar{H}$
le morphisme naturel. Comme $q$ est un fibr\'e localement trivial,
on peut trouver une section rationnelle
$\beta:\tilde{C}\dashrightarrow\Sigma$ de $\alpha$;
on peut de plus supposer que son image n'est pas incluse dans $E$.
Par crit\`ere valuatif de propret\'e, $\beta$ est en fait un
morphisme. Notons $\gamma=e\circ\beta$. Comme
$\beta(\tilde{C})\not\subset E$, on a $\gamma(\tilde{C})\not\subset W$.
On peut donc choisir par la proposition \ref{eqZ} un diviseur de Cartier
$D\in|\mathcal{O}_{\bar{H}_{d_1}\times\bar{H}_{d_2}^{c-1}}((c-1)(d_2-d_1)+1,1,\dots,1)|$
contenant $W$ mais pas $\gamma(\tilde{C})$, donc tel que $e^*D$ contienne $E$ mais pas $\beta(\tilde{C})$. On calcule alors :
\begin{alignat*}{7}
   \tilde{C}\cdot\alpha^*\mathcal{O}_{\bar{H}}&(l_1,l_2)  =
\tilde{C}\cdot\beta^*q^*\mathcal{O}_{\bar{H}}(l_1,l_2) &&\text{ par projection}\\
      & = \tilde{C}\cdot\beta^*(e^*\mathcal{O}_{\bar{H}_{d_1}\times\bar{H}_{d_2}^{c-1}}(l_1,l_2,\dots,l_2)-l_2E)
&&\text{ par \ref{E}}\\
  & \geq \tilde{C}\cdot\beta^*(e^*\mathcal{O}_{\bar{H}_{d_1}\times\bar{H}_{d_2}^{c-1}}(l_1,l_2,\dots,l_2)-l_2e^*D) &&\text{ car }E\subset e^*D\text{, }l_2\geq0\\
&=\tilde{C}\cdot\gamma^*\mathcal{O}(l_1-l_2((c-1)(d_2-d_1)+1),0,\dots,0)
&&\text{ par projection }\\
&\geq0.&&
\end{alignat*}
On a bien montr\'e que $\mathcal{O}_{\bar{H}}(l_1,l_2)$ est nef.
\end{proof}

\begin{rem}\label{echec}
Au paragraphe \ref{edmaff}, on a pu montrer facilement le th\'eor\`eme \ref{edm} (i) car le diviseur discriminant
$\Delta=\bar{H}\setminus H$ \'etait ample sur $\bar{H}$, et son compl\'e\-men\-taire $H$ \'etait donc affine.
Cette m\'ethode ne permet pas de montrer le th\'eor\`eme \ref{edm} (ii) ; plus pr\'ecis\'ement, si $c=2$, elle ne fonctionne jamais.

 En effet, le th\'eor\`eme 1.2 de \cite{Ooldeg} permet de calculer le fibr\'e en droites associ\'e au diviseur discriminant
$\Delta=\bar{H}\setminus H$. Quand $c=2$, les calculs sont men\'es dans l'exemple 1.10 de \cite{Ooldeg}, et on obtient 
$\mathcal{O}(\Delta)=\mathcal{O}(l_1,l_2)$ avec
$l_1=d_2(e_2^{N-1}+2e_1e_2^{N-2}+\dots+Ne_1^{N-1})$ et $
l_2=d_1(e_1^{N-1}+2e_2e_1^{N-2}+\dots+Ne_2^{N-1}) $, et o\`u l'on a
pos\'e $e_i=d_i-1$. Comme $\frac{l_1}{l_2}\leq\frac{d_2}{d_1}\leq
d_2-d_1+1$, le th\'eor\`eme \ref{ample} montre que ce fibr\'e n'est
jamais ample.

Quand $c>2$, les formules calculant $l_1$ et $l_2$ sont plus compliqu\'ees, et font appara\^itre des sommes altern\'ees,
ce qui rend difficile une v\'erification analogue.
\end{rem}

\subsection{Th\'eorie g\'eom\'etrique des invariants}\label{preuveedm}

Dans ce paragraphe, on applique le crit\`ere de Hilbert-Mumford pour montrer le th\'eor\`eme \ref{edm} (ii).
On commence par \'evaluer les fonctions $\mu$ intervenant dans ce crit\`ere
pour l'action de $SL_{N+1}$ sur $\bar{H}$ relativement aux fibr\'es en droites $SL_{N+1}$-lin\'earis\'es 
d\'ecrits au paragraphe \ref{constructionsqp}.
Ces fonctions $\mu$ d\'ependent d'un point g\'eom\'etrique $P=[F_1,F_2,\dots, F_c]\in\bar{H}(K)$ et d'un sous-groupe \`a un 
param\`etre non trivial $\rho : \mathbb{G}_{m,K}\to SL_{N+1,K}$. 

Rappelons bri\`evement leur d\'efinition. Consid\'erons la fibre en $\lim_{t\to
0}\rho(t)\cdot P$ du fibr\'e en droites g\'eom\'etrique sur $\bar{H}$ associ\'e \`a $\mathcal{O}(l_1,l_2)$.
Le morphisme $\rho$ induit une action de $\mathbb{G}_{m,K}$ sur cette fibre.
Cette action se fait via un caract\`ere de $\mathbb{G}_{m,K}$,
c'est-\`a-dire un entier ; on note $\mu^{\mathcal{O}(l_1,l_2)}(P,\rho)$
l'oppos\'e de cet entier.
Dans les deux lemmes qui suivent, on met $\rho$ et $P$ sous une forme
qui permettra de calculer $\mu^{\mathcal{O}(l_1,l_2)}(P,\rho)$.

\begin{lemme}\label{bon1ps}
Soit $\rho : \mathbb{G}_{m,K}\to SL_{N+1,K}$ un sous-groupe \`a un param\`etre non trivial.
Alors on peut trouver des entiers $\alpha_0\leq\dots\leq\alpha_N$ non tous nuls 
de somme nulle et une base de $K^{N+1}$ dans laquelle
$\rho(t)\cdot(x_0,\dots,x_N)=(t^{\alpha_0}x_0,\dots,t^{\alpha_N}x_N)$.
\end{lemme}

\begin{proof}[$\mathbf{Preuve}$]
C'est standard.
\end{proof}

Dans le reste de ce paragraphe, $\rho$ est fix\'e. On travaille avec un syst\`eme de coordonn\'ees et
des entiers $\alpha_i$ comme dans le lemme \ref{bon1ps}.

\begin{conventions}\label{notationsalpha}
Si $\alpha$ est la donn\'ee d'entiers
$\alpha_0\leq\dots\leq\alpha_N$ non tous nuls de somme nulle, le
$\alpha$-degr\'e d'un mon\^ome $M=X_0^{\lambda_0}\dots
X_N^{\lambda_N}$ est $\deg_{\alpha}(M)=\sum_i\alpha_i\lambda_i$.
Si $F\in
H^0(\mathbb{P}^N_K,\mathcal{O}(d))$ est une \'equation non nulle, on
note $\deg_{\alpha}(F)$ le plus grand $\alpha$-degr\'e des mon\^omes de $F$. Par convention, 
$\deg_\alpha(0)=-\infty$. Soit $F^{\alpha}$ la somme des termes de $F$ de $\alpha$-degr\'e
$\deg_{\alpha}(F)$. On dit que $F$ est $\alpha$-homog\`ene si $F=F^\alpha$.
\end{conventions}

\begin{lemme}\label{bonneq}
Soit $P=[F_1,F_2,\dots,F_c]\in\bar{H}(K)$.
Alors il existe des \'equations $\Phi_i\in H^0(\mathbb{P}^N_K,\mathcal{O}(d_i))$ pour $2\leq i\leq c$ telles que :
\begin{enumerate}[(i)]
 \item $P=[F_1,\Phi_2,\dots,\Phi_c]$.
\item $\deg_{\alpha}(\Phi_i)\leq\deg_{\alpha}(F_i)$ pour $2\leq i\leq c$.
\item    $[F_1^{\alpha},\Phi_2^{\alpha},\dots, \Phi_c^{\alpha}]\in\bar{H}(K)$.
\end{enumerate}
\end{lemme}

\begin{proof}[$\mathbf{Preuve}$]

Choisissons $\Phi_i\in H^0(\mathbb{P}^N_K,\mathcal{O}(d_i))$ pour $2\leq i\leq c$ v\'erifiant les propri\'et\'es (i) et (ii), et telles que 
$\sum_{i=2}^c\deg_{\alpha}(\Phi_i)$ soit minimal.
%C'est possible car on a la minoration $\deg_{\alpha}(\Phi_i)\geq d_i\alpha_0$. 
Montrons par l'absurde que la condition (iii) est automatiquement v\'erifi\'ee.

Si ce n'\'etait pas le cas, on pourrait trouver $Q\in H^0(\mathbb{P}^N_K,\mathcal{O}(d_2-d_1))$ et
$\lambda_1,\dots,\lambda_c\in K$ non tous nuls tels que 
$QF_1^{\alpha}=\sum_{i=2}^c\lambda_i \Phi_i^{\alpha}$. En ne gardant dans cette identit\'e que les termes de $\alpha$-degr\'e maximal (i.e.
quitte \`a remplacer $Q$ par $Q^\alpha$ ou $0$ et \`a remplacer certains des $\lambda_i$ par $0$), 
on peut supposer que tous les termes
de cette identit\'e sont $\alpha$-homog\`enes de m\^eme $\alpha$-degr\'e.
Soit alors $2\leq j\leq c$ tel que $\lambda_{j}$ soit non nul ; on pose $\Phi'_i=\Phi_i$ si $i\neq j$ et
$\Phi'_j=\sum_{i=2}^c\lambda_i \Phi_i-QF_1$.

Les $\Phi'_i$ v\'erifient encore la propri\'et\'e (i). On a bien $\deg_\alpha(\Phi'_i)=\deg_\alpha(\Phi_i)$ si $i\neq j$.
De plus, l'expression de $\Phi'_j$ montre que $\deg_{\alpha}(\Phi'_j)\leq\deg_{\alpha}(\Phi_j)$, 
mais que la somme des termes de $\alpha$-degr\'e $\deg_{\alpha}(\Phi_j)$ dans $\Phi'_j$
est nulle, i.e. $\deg_{\alpha}(\Phi'_j)<\deg_{\alpha}(\Phi_j)$. D'une part cela montre que les $\Phi'_i$ v\'erifient encore la propri\'et\'e (ii).
D'autre part, cela contredit la minimalit\'e dans le choix des $\Phi_i$.
\end{proof}

Calculons maintenant les fonctions $\mu$ qui nous seront utiles.

\begin{lemme}\label{mu1}
Soit $\langle F_1\rangle\in\bar{H}_{d_1}(K)$.
Alors $\mu^{\mathcal{O}(1)}(\langle F_1\rangle,\rho)=\deg_{\alpha}(F_1)$.
\end{lemme}

\begin{proof}[$\mathbf{Preuve}$]
On rappelle que, par d\'efinition de l'action duale,
 si $F$ est un \'el\'ement $\alpha$-homog\`ene de $H^0(\mathbb{P}^N_K,\mathcal{O}(d_1))=\Sym^{d_1}(K^{N+1})^{\vee}$, l'action de $\rho$ sur $F$ est donn\'ee
par $\rho(t)\cdot F=t^{-\deg_{\alpha}(F)}F$. Ainsi, si l'on \'ecrit $F_1=F_1^{\alpha}+R$,
\begin{alignat*}{2}
 \rho(t)\cdot\langle F_1\rangle=\langle\rho(t)\cdot F_1\rangle&=\langle t^{\deg_{\alpha}(F_1)}\rho(t)\cdot F_1\rangle&&
\\
      &= \langle F_1^{\alpha}+t^{\deg_{\alpha}(F_1)}\rho(t)\cdot R \rangle.&&
\end{alignat*}
Le terme de droite tendant vers $0$, on a $\lim_{t\to
0}\rho(t)\cdot\langle F_1\rangle=\langle F_1^{\alpha}\rangle$.

Enfin, dans $H^0(\mathbb{P}^N_K,\mathcal{O}(d_1))$,
$\rho(t)\cdot F_1^{\alpha}=t^{-\deg_{\alpha}(F_1)}F_1^{\alpha}$,
ce qui montre, par d\'efinition de la $SL_{N+1}$-lin\'earisation de $\mathcal{O}(1)$, que
$\mu^{\mathcal{O}(1)}(\langle F_1\rangle,\rho)=\deg_{\alpha}(F_1)$.
\end{proof}

\begin{lemme}\label{mu2}
Soit $P\in\bar{H}(K)$. On \'ecrit $P=[F_1,\Phi_2,\dots,\Phi_c]$ o\`u les $\Phi_i$ ont \'et\'e choisis comme dans le lemme \ref{bonneq}. Alors
$\mu^{\mathcal{O}(0,1)}(P,\rho)=\sum_{i=2}^c\deg_{\alpha}(\Phi_i)$.
\end{lemme}

\begin{proof}[$\mathbf{Preuve}$]
La preuve est analogue \`a celle du lemme pr\'ec\'edent.
\end{proof}

En combinant les lemmes \ref{mu1} et \ref{mu2}, on obtient :

\begin{prop}\label{mucombo}
Soit $P\in\bar{H}(K)$. On \'ecrit $P=[F_1,\Phi_2,\dots,\Phi_c]$ o\`u les $\Phi_i$ ont \'et\'e choisis comme dans le lemme \ref{bonneq}. Alors :
$$\mu^{\mathcal{O}(l_1,l_2)}(P,\rho)=l_1\deg_{\alpha}(F_1)+l_2\sum_{i=2}^c\deg_{\alpha}(\Phi_i).$$
\end{prop}

Nous sommes pr\^ets \`a appliquer le crit\`ere de Hilbert-Mumford.

\begin{prop}\label{stablisse}
Il existe un fibr\'e en droites ample $SL_{N+1}$-lin\'earis\'e $\mathcal{L}$ sur $\bar{H}$ tel
que $H\subset\bar{H}^s(\mathcal{L})$ si et seulement si 
\begin{equation}\label{condinum}
d_2(N-c+2)>d_1((c-1)(d_2-d_1)+1).
\end{equation}
\end{prop}

\begin{proof}[$\mathbf{Preuve}$]
On a vu au paragraphe \ref{constructionsqp} que les fibr\'es en droites sur $\bar{H}$
sont de la forme $\mathcal{O}(l_1,l_2)$ et sont uniquement $SL_{N+1}$-lin\'earis\'es.
Par le th\'eor\`eme \ref{ample}, un tel fibr\'e en droites 
est ample si et seulement si $l_2>0$ et $\frac{l_1}{l_2}>(c-1)(d_2-d_1)+1$.

Supposons dans un premier temps que (\ref{condinum}) est v\'erifi\'ee 
et montrons que $\mathcal{L}=\mathcal{O}(l_1,l_2)$ avec $l_1=kd_2(N+2-c)-1$ et $l_2=kd_1$ convient si $k\gg0$.
Ce fibr\'e en droites est bien ample : $l_2>0$ et $\frac{l_1}{l_2}>(c-1)(d_2-d_1)+1$ est vrai pour $k\gg0$ par
(\ref{condinum}). Montrons alors
$H\subset\bar{H}^s(\mathcal{O}(l_1,l_2))$ en appliquant le crit\`ere de Hilbert-Mumford (\cite{GIT} Theorem 2.1).
Soient pour cela $P=[F_1,F_2,\dots, F_c]\in H(K)$ et
$\rho : \mathbb{G}_{m,K}\to SL_{N+1,K}$ un sous-groupe \`a un 
param\`etre non trivial qu'on peut supposer de la forme obtenue dans le lemme \ref{bon1ps}. 
Par le lemme \ref{bonneq} et la proposition \ref{mucombo}, quitte \`a modifier $F_2,\dots,F_c$, on peut supposer que 
$\mu^{\mathcal{O}(l_1,l_2)}(P,\rho)=l_1\deg_{\alpha}(F_1)+l_2\sum_{i=2}^c\deg_{\alpha}(F_i).$ Par le th\'eor\`eme \ref{alphadeg} (ii),
pour montrer que $\mu^{\mathcal{O}(l_1,l_2)}(P,\rho)>0$ et conclure, il suffit
de v\'erifier que $(N+1)l_1d_1>l_1d_1+(c-1)l_2d_2$ et que $(N+1)l_2d_2>l_1d_1+(c-1)l_2d_2$, i.e. que :
$$\frac{d_2}{d_1}\frac{c-1}{N}<\frac{l_1}{l_2}<\frac{d_2}{d_1}(N-c+2).$$
On remarque alors que $\frac{l_1}{l_2}$
est une fonction croissante de $k$ qui tend vers $\frac{d_2}{d_1}(N-c+2)$. Cela montre que la seconde in\'egalit\'e est toujours vraie. 
Comme $\frac{c-1}{N}<1<N-c+2$, cela montre aussi que la premi\`ere in\'egalit\'e est vraie pour $k\gg0$.

R\'eciproquement, supposons que (\ref{condinum}) n'est pas v\'erifi\'ee, et soit $\mathcal{O}(l_1,l_2)$ un fibr\'e ample sur $\bar{H}$.
L'amplitude de 
$\mathcal{O}(l_1,l_2)$ et le fait que (\ref{condinum}) n'est pas vraie montrent que $\frac{l_1}{l_2}\geq\frac{d_2}{d_1}(N-c+2)$, donc que 
$(N+1)l_2d_2\leq l_1d_1+(c-1)l_2d_2$. Alors, par le th\'eor\`eme \ref{alphadeg} (iii),
on peut trouver des entiers $\alpha_0\leq\dots\leq\alpha_N$
non tous nuls de somme nulle et des \'equations 
$F_i\in H^0(\mathbb{P}^N_K,\mathcal{O}(d_i))$, $1\leq i\leq c$, d\'efinissant une intersection compl\`ete lisse telles que
$l_1\deg_{\alpha}(F_1)+l_2\sum_{i=2}^c\deg_{\alpha}(F_i)\leq 0$.
Soient $\Phi_2,\dots,\Phi_c$ comme dans le lemme \ref{bonneq}. Par la condition (ii) de ce lemme, on a encore
$l_1\deg_{\alpha}(F_1)+l_2\sum_{i=2}^c\deg_{\alpha}(\Phi_i)\leq 0$.
Notons $\rho:\mathbb{G}_{m,K}\to SL_{N+1,K}$ le sous-groupe \`a un param\`etre d\'efini par
$\rho(t)\cdot(x_0,\dots,x_N)=(t^{\alpha_0}x_0,\dots,t^{\alpha_N}x_N)$, et $P=[F_1,F_2,\dots,F_c]\in H(K)$. Par la proposition
\ref{mucombo}, $\mu^{\mathcal{O}(l_1,l_2)}(P,\rho)\leq 0$, et le
crit\`ere de Hilbert-Mumford montre que $P\notin\bar{H}^s(\mathcal{O}(l_1,l_2))$. 
\end{proof}

On en d\'eduit imm\'ediatement le th\'eor\`eme \ref{edm}(ii).

\begin{proof}[$\mathbf{Preuve \text{ }du \text{ }th\acute{e}or\grave{e}me\text{ }\ref{edm} (ii)}$]~
Par la proposition \ref{stablisse}, il existe un fibr\'e en droites $SL_{N+1}$-lin\'earis\'e ample sur $\bar{H}$ rendant tous les
points de $H$ stables. La th\'eorie g\'eom\'etrique des invariants permet donc de construire un quotient g\'eom\'etrique quasi-projectif de $H$ par
$SL_{N+1}$ (\cite{Seshadri} Theorem 4). Celui-ci
est aussi un quotient g\'eom\'etrique quasi-projectif de $H$ par $PGL_{N+1}$ : l'espace de modules grossier $M$ 
est donc bien quasi-projectif.
\end{proof}

\section{Minoration du $\alpha$-degr\'e}

On fixe un corps alg\'ebriquement clos $K$.
Dans cette partie, on autorise  $1\leq c\leq N$ et $2\leq d_1\leq\dots\leq d_c$.
Une intersection compl\`ete sur $K$ est toujours de codimension $c$
dans $\mathbb{P}^N_K$ et de degr\'es $d_1,\dots,d_c$.
On conserve les conventions \ref{notationsalpha}.
L'objectif de cette partie est la preuve de l'in\'egalit\'e suivante, qu'on a utilis\'ee
au paragraphe \ref{preuveedm} pour v\'erifier le crit\`ere de Hilbert-Mumford.

\begin{thm}\label{alphadeg}
\begin{enumerate}[(i)]
 \item 
Soient $k_1,\dots, k_c$ des nombres r\'eels tels que :
\begin{equation}\label{hypki}
\min_{1\leq i\leq c}k_i\geq \frac{1}{N+1}\sum_{i=1}^c k_i.
\end{equation}
Alors si $\alpha_0\leq\dots\leq\alpha_N$ sont des entiers non tous nuls de somme nulle et si $F_1,\dots, F_c$ constituent
une suite r\'eguli\`ere globale d\'efinissant une intersection compl\`ete lisse, on a :
\begin{equation}\label{concki}\sum_{i=1}^c k_i\frac{\deg_{\alpha}(F_i)}{d_i}\geq 0.
\end{equation}

\item Supposons qu'on n'a pas $c=1$ et $d_1=2$. Alors si l'in\'egalit\'e (\ref{hypki}) est stricte, l'in\'egalit\'e (\ref{concki}) est stricte.

\item Les \'enonc\'es (i) et (ii) sont optimaux au sens o\`u ils seraient faux pour d'autres valeurs des $k_i$.
\end{enumerate}
\end{thm}

Pr\'ecisons le sens de (iii). Dire que l'\'enonc\'e (i) est optimal signifie que si $k_1,\dots, k_c$
sont des r\'eels ne v\'erifiant pas (\ref{hypki}), il existe des
entiers non tous nuls de somme nulle $\alpha_0\leq\dots\leq\alpha_N$ et des \'equations $F_1,\dots, F_c$
d\'efinissant une intersection compl\`ete lisse tels que l'in\'egalit\'e (\ref{concki}) soit fausse.
L'assertion concernant l'\'enonc\'e (ii) est analogue.

L'in\'egalit\'e \ref{alphadeg} permet de minorer
les $\alpha$-degr\'es des \'equations d'une intersection compl\`ete lisse.
Son heuristique est la suivante : si les $\alpha$-degr\'es des \'equations d'une intersection compl\`ete sont petits,
cela signifie que beaucoup de mon\^omes n'interviennent pas dans ces \'equations. Ce fait doit permettre
de montrer, via le crit\`ere jacobien, que cette intersection compl\`ete est en fait singuli\`ere.

Le paragraphe \ref{uneequation} est constitu\'e de r\'esultats pr\'eliminaires autour du lien entre $\alpha$-degr\'e
d'une \'equation $F$ et singularit\'es
de $\{F=0\}$ ; le paragraphe \ref{pleindequations} est consacr\'e \`a la preuve du th\'eor\`eme \ref{alphadeg}.

\subsection{\'Etude d'une \'equation}\label{uneequation}

On fixe dans tout ce paragraphe un entier $d\geq2$ et une \'equation non nulle $F\in H^0(\mathbb{P}^N_K,\mathcal{O}(d))$.
Le lemme \ref{singdeg} permettra de faire le lien entre la g\'eom\'etrie de l'hypersurface $\{F=0\}$ et le  $\alpha$-degr\'e $\deg_{\alpha}(F)$.

\begin{lemme}\label{singdeg}
Soient $u$, $v$ et $s$ des entiers tels que $u,v\geq0$, $s\geq0$ et $u+v=N-s$.
Alors, si $\deg_{\alpha}(F)<\alpha_u+(d-1)\alpha_v$, $$\dim(\Sing(\{F=0\})\cap\{X_0=\dots=X_{v-1}=0\})\geq s.$$
\end{lemme}

\begin{proof}[$\mathbf{Preuve}$]
Comme $d\geq 2$ et les $\alpha_i$ sont croissants, quitte \`a \'echanger $u$ et $v$, on peut supposer que $u\leq v$.
\'Ecrivons alors $F=X_0P_0+\dots+X_NP_N$, o\`u $P_i$ ne d\'epend
pas de $X_0,\dots,X_{i-1}$. L'hypoth\`ese
$\deg_{\alpha}(F)<\alpha_u+(d-1)\alpha_v$ montre que si $i\geq u$,
$P_i$ ne d\'epend que de $X_0,\dots, X_{v-1}$.

Posons $Z=\{X_0=\dots=X_{v-1}=P_0=\dots=P_{u-1}=0\}$. Si $i\leq v-1$, $X_i$ est nul sur $Z$.
Si $i\geq v$, on a $i\geq u$, de sorte que $P_i$, qui ne d\'epend que de $X_0,\dots, X_{v-1}$, est nul sur $Z$.
Par cons\'equent, $F=\sum_{i=0}^N X_i P_i$ est nul sur $Z$. 
De m\^eme, pour $0\leq j\leq N$, on peut \'ecrire 
$\frac{\partial F}{\partial X_j}=P_j+\sum_{i=0}^NX_i\frac{\partial P_i}{\partial X_j}$. En distinguant comme ci-dessus les cas $i\leq v-1$
et $i\geq v$, on voit que $X_i\frac{\partial P_i}{\partial X_j}$ s'annule sur $Z$. De plus, si $j<u$, $P_j$ est nul sur $Z$ et si $j\geq u$,
$P_j$ qui ne d\'epend que de $X_0,\dots, X_{v-1}$, est aussi nul sur $Z$.
Sommant, on voit que $\frac{\partial F}{\partial X_j}$ est nul sur $Z$.

On a montr\'e que $F$ et tous les
$\frac{\partial F}{\partial X_j}$ s'annulent sur $Z$, de sorte que,
par le crit\`ere jacobien, $Z\subset\Sing(\{F=0\})$. Comme, par le
th\'eor\`eme de l'intersection projective, $\dim(Z)\geq N-u-v=s$,
le lemme est d\'emontr\'e.
\end{proof}

Le lemme \ref{singdeg} motive la d\'efinition qui suit.

\begin{defi}
On note $s(F)$ le plus petit entier $s\in\{-1,\dots,N-1\}$ tel que, si $u,v\geq 0$ sont des entiers avec $u+v=N-s-1$, on a
$\deg_{\alpha}(F)\geq\alpha_u+(d-1)\alpha_v$.

Supposons que $0\leq s \leq s(F)$.
On note $v_s(F)$ le plus grand entier $v\in\{0,\dots,N-s\}$ tel que $\deg_{\alpha}(F)<\alpha_{N-v_s(F)-s}+(d-1)\alpha_{v_s(F)}$.
\end{defi}

\begin{lemme}\label{vFex}
L'entier $s(F)$ est bien d\'efini.

Soit $0\leq s \leq s(F)$. Alors $v_s(F)$ est bien d\'efini et $v_s(F)\geq \frac{N+s(F)-2s}{2}$.
De plus, si $0<s\leq s(F)$, on a $v_{s-1}(F)\geq v_{s}(F)+1$.
\end{lemme}

\begin{proof}[$\mathbf{Preuve}$]
Comme le plus petit $\alpha$-degr\'e possible d'un mon\^ome de degr\'e $d$ est $d\alpha_0$, 
on a $\deg_\alpha(F)\geq d\alpha_0$, de sorte que $s(F)$ est bien d\'efini.

Par d\'efinition de $s(F)$, et comme $s(F)\geq 0$, il existe $u,v\geq 0$ tels que $u+v=N-s(F)$ et
$\deg_{\alpha}(F)<\alpha_u+(d-1)\alpha_v$. Comme $d\geq 2$ et que les $\alpha_i$ sont croissants, quitte \`a 
\'echanger $u$ et $v$, on peut supposer $v\geq u$, soit $v\geq \frac{N-s(F)}{2}$. Alors, si $v'=v+s(F)-s$, comme les $\alpha_i$
sont croissants, on a $\deg_{\alpha}(F)<\alpha_u+(d-1)\alpha_{v'}$.
Ceci montre d'une part l'existence de $v_s(F)$ et d'autre part que $v_s(F)\geq v'\geq\frac{N-s(F)}{2}+s(F)-s=\frac{N+s(F)-2s}{2}$.

Finalement, supposons $0<s\leq s(F)$. Comme $\deg_{\alpha}(F)<\alpha_{N-s-v_{s}(F)}+(d-1)\alpha_{v_{s}(F)}$,
par croissance des $\alpha_i$, on obtient
$\deg_{\alpha}(F)<\alpha_{N-s-v_s(F)}+(d-1)\alpha_{v_{s}(F)+1}$, ce qui montre que $v_{s-1}(F)\geq v_{s}(F)+1$.
\end{proof}

Le lemme ci-dessous permet d'interpr\'eter
l'entier $s(F)$ comme la dimension attendue, connaissant $\deg_{\alpha}(F)$, de $\Sing(\{F=0\})$.
Les entiers $v_s(F)$ indiquent, eux, la mani\`ere dont on s'attend
\`a ce que les singularit\'es de $\{F=0\}$
se situent par rapport au drapeau
$\varnothing\subset\{X_0=\dots=X_{N-1}=0\}\subset\dots\subset\{X_0=0\}\subset\mathbb{P}^N_K$.

\begin{lemme}\label{lemsing}
On a $\dim(\Sing(\{F=0\}))\geq s(F)$.

Si $0\leq s \leq s(F)$,
$\dim(\Sing(\{F=0\})\cap\{X_0=\dots=X_{v_s(F)-1}=0\})\geq s$.
\end{lemme}

\begin{proof}[$\mathbf{Preuve}$]
C'est une cons\'equence imm\'ediate du lemme \ref{singdeg} et des d\'efinitions.
\end{proof}

Les deux lemmes suivants permettent de minorer le $\alpha$-degr\'e des \'equations telles que $s(F)=-1$ (resp. $s(F)\geq 0$).

\begin{lemme}\label{alisse}
Supposons que $s(F)=-1$. Alors
$\deg_{\alpha}(F)\geq 0$.

De plus, cette in\'egalit\'e est stricte si $d\geq 3$.
\end{lemme}

\begin{proof}[$\mathbf{Preuve}$]
On calcule :
\begin{alignat*}{3}
  N\deg_{\alpha}(F) &\geq(\alpha_0+(d-1)\alpha_N)+\dots+(\alpha_{N-1}+(d-1)\alpha_1) &&
\text{ car }s(F)=-1 \\
      &= -\alpha_N-(d-1)\alpha_0 &&\text{ car }\sum_i\alpha_i=0\\
 \frac{1}{d-1}\deg_{\alpha}(F) & \geq\frac{1}{d-1}\alpha_0+\alpha_N&&\text{ car }s(F)=-1.
\end{alignat*}
Sommant ces deux in\'egalit\'es, on obtient
$\deg_{\alpha}(F)\geq-\frac{d(d-2)}{Nd-N+1}\alpha_0$.
Ceci conclut car $d\geq 2$ et $\alpha_0<0$ (les $\alpha_i$ sont croissants non tous nuls de somme nulle).
\end{proof}

\begin{lemme}\label{nonalisse}
Supposons que $s(F)\geq 0$. Alors :
 $$\frac{\deg_{\alpha}(F)}{d}\geq-\frac{\sum_{s=0}^{s(F)}\alpha_{v_s(F)}}{N-s(F)}.$$
\end{lemme}

\begin{proof}[$\mathbf{Preuve}$]
Par d\'efinition de $v_0(F)$, on a :
\begin{alignat}{2}
(N-v_0(F))\deg_{\alpha}(F) &\geq(\alpha_0+(d-1)\alpha_N) && \nonumber \\
      &+\dots+(\alpha_{N-v_0(F)-1}+(d-1)\alpha_{v_0(F)+1}).&&
\label{in1}
\end{alignat}
Pour $0< s\leq s(F)$, par d\'efinition de $v_s(F)$, et comme $v_{s-1}(F)\geq v_{s}(F)+1$ par le lemme \ref{vFex}, on a :
\begin{alignat}{2}
(v_{s-1}(F)-v_s(F)-1)&\deg_{\alpha}(F) \geq (\alpha_{N-s-v_{s-1}(F)+1}+(d-1)\alpha_{v_{s-1}(F)-1}) && \nonumber \\
      &+\dots+(\alpha_{N-s-v_s(F)-1}+(d-1)\alpha_{v_s(F)+1}).&&
\label{in2}
\end{alignat}
Le lemme \ref{vFex} montre que $2v_{s(F)}(F)+s(F)-N\geq 0$.
Ceci permet d'\'ecrire, utilisant la d\'efinition de $s(F)$ :
\begin{alignat}{2}
(2v_{s(F)}+s(F)-N)&\deg_{\alpha}(F) \geq(\alpha_{N-s(F)-v_{s(F)}(F)}+(d-1)\alpha_{v_{s(F)}(F)-1}) && \nonumber \\
      &+\dots+(\alpha_{v_{s(F)}(F)-1}+(d-1)\alpha_{N-s(F)-v_{s(F)}(F)}).&&
\label{in3}
\end{alignat}
Sommant deux fois l'in\'egalit\'e (\ref{in1}), deux fois les in\'egalit\'es (\ref{in2}) et l'in\'egalit\'e (\ref{in3}), on obtient :
\begin{alignat}{3}
(N-s(F))&\deg_{\alpha}(F) \geq2(\alpha_0+\dots+\alpha_{N-s(F)-v_{s(F)}(F)-1})&& \nonumber \\
&+d(\alpha_{N-s(F)-v_{s(F)}(F)}+\dots+\alpha_{v_{s(F)}(F)-1}) &&\label{in4}  \\
      &+
(2d-2)(\alpha_{v_{s(F)}(F)}+\dots+\alpha_{N})-(2d-2)\sum_{s=0}^{s(F)}\alpha_{v_s(F)}.&&\nonumber
\end{alignat}
Remarquons alors que : 
$$0\geq[\alpha_0+\dots+\alpha_{N-s(F)-v_{s(F)}(F)-1}]-[\alpha_{v_{s(F)}(F)}+\dots+\alpha_{N}-\sum_{s=0}^{s(F)}\alpha_{v_s(F)}].$$
En effet, chacun des crochets est une somme de $N-s(F)-v_{s(F)}(F)$ des $\alpha_i$. Les indices intervenant dans le premier crochet sont tous plus petits que
les indices intervenant dans le second, de sorte que l'on conclut par croissance des $\alpha_i$.

Multipliant cette \'equation par $(d-2)\geq 0$, et l'ajoutant \`a (\ref{in4}), on obtient :
$$(N-s(F))\deg_{\alpha}(F) \geq d(\alpha_{0}+\dots+\alpha_{N})-d\sum_{s=0}^{s(F)}\alpha_{v_s(F)}.$$
Comme les $\alpha_i$ sont de somme nulle, cela prouve l'in\'egalit\'e voulue.
\end{proof}

Finalement, montrons une propri\'et\'e de positivit\'e des $\alpha_{v_s(F)}$ :

\begin{lemme}\label{positivite}
Supposons que $s(F)\geq 0$. Alors : $$\alpha_{v_{s(F)}(F)}+\frac{\sum_{s=0}^{s(F)}\alpha_{v_s(F)}}{N-s(F)}>0.$$
\end{lemme}

\begin{proof}[$\mathbf{Preuve}$]
En sommant les in\'egalit\'es (\ref{in1}) et (\ref{in2}) de la preuve du lemme \ref{nonalisse}, on obtient :
\begin{alignat}{2}
(N-s(F)-v_{s(F)}&(F))\deg_{\alpha}(F) \geq(\alpha_0+\dots+\alpha_{N-s(F)-v_{s(F)}(F)-1})&& \nonumber \\
      &+
(d-1)(\alpha_{v_{s(F)}(F)}+\dots+\alpha_{N})-(d-1)\sum_{s=0}^{s(F)}\alpha_{v_s(F)}.&&
\label{in5}
\end{alignat}

Par d\'efinition de $v_{s(F)}(F)$, $\deg_{\alpha}(F)<\alpha_{N-s(F)-v_{s(F)}(F)}+(d-1)\alpha_{v_{s(F)}(F)}$. Comme, par le lemme \ref{vFex},
$N-s(F)-v_{s(F)}(F)\leq v_{s(F)}(F)$, la croissance des $\alpha_i$ montre $\deg_{\alpha}(F)<d \alpha_{v_{s(F)}(F)}$. Combinons ce fait avec
l'in\'egalit\'e (\ref{in5}), puis utilisons le fait que $d\geq 2$ et que les $\alpha_i$ sont croissants.
\begin{alignat*}{4}
(N-&s(F)&&-v_{s(F)}(F))d\alpha_{v_{s(F)}(F)}+(d-1)\sum_{s=0}^{s(F)}\alpha_{v_s(F)} &&&\nonumber \\
& >&&(\alpha_0+\dots+\alpha_{N-s(F)-v_{s(F)}(F)-1})+(d-1)(\alpha_{v_{s(F)}(F)}+\dots+\alpha_{N}) &&&\nonumber \\
&\geq&&(d-1)(\alpha_0+\dots+\alpha_{N-s(F)-v_{s(F)}(F)-1}+\alpha_{v_{s(F)}(F)}+\dots+\alpha_{N}) &&&\nonumber \\
& &&-(N-s(F)-v_{s(F)}(F))(d-2)\alpha_{v_{s(F)}(F)}.&&&\nonumber
\end{alignat*}
Utilisons que les $\alpha_i$ sont de somme nulle, puis \`a nouveau leur croissance :
\begin{alignat*}{4}
(N-&s(F)-v_{s(F)}(F))(2d-2)\alpha_{v_{s(F)}(F)}+(d-1)\sum_{s=0}^{s(F)}\alpha_{v_s(F)} &&\nonumber \\
&>-(d-1)(\alpha_{N-s(F)-v_{s(F)}(F)}+\dots+\alpha_{v_{s(F)}(F)-1}) &&\nonumber\\
&\geq -(2v_{s(F)}(F)+s(F)-N)(d-1)\alpha_{v_{s(F)}(F)}. &&\nonumber
\end{alignat*}
On obtient l'in\'egalit\'e voulue en divisant par $(d-1)(N-s(F))>0$.
\end{proof}

\subsection{\'Equations d'une intersection compl\`ete lisse}\label{pleindequations}

Utilisons les r\'esultats du paragraphe pr\'ec\'edent pour montrer le th\'eor\`eme \ref{alphadeg}.

\begin{proof}[$\mathbf{Preuve \text{ }du \text{ }th\acute{e}or\grave{e}me\text{ }\ref{alphadeg}\text{ }(i)}$]~

  Tout d'abord, en sommant pour $i\in \{1,\dots,c\}$ les in\'egalit\'es $k_i\geq\frac{k_1+\dots+k_c}{N+1}$, 
on montre $(N+1-c)(k_1+\dots+k_c)\geq0$, donc $k_1+\dots+k_c\geq0$, et finalement,
$k_i\geq0$ pour $i\in \{1,\dots, c\}$.

Distinguons alors deux cas. Si $s(F_i)=-1$ pour tout $i\in \{1,\dots,c\}$,
le lemme \ref{alisse} montre que 
$\deg_{\alpha}(F_i)\geq 0$. Ainsi, $\sum_{i=1}^c k_i\frac{\deg_{\alpha}(F_i)}{d_i}\geq0$.

Supposons au contraire qu'il existe $l$ tel que $s(F_l)\geq 0$. 
On choisit un tel $l$ de sorte que $\frac{\sum_{s=0}^{s(F_l)}\alpha_{v_s(F_l)}}{N-s(F_l)}$ soit maximal. On va construire des entiers
$j_0,\dots, j_{s(F_l)}\in\{1,\dots,c\}$ distincts tels que, pour $0\leq s \leq s(F_l)$,  
\begin{equation}
\frac{\deg_{\alpha}(F_{j_s})}{d_{j_s}}\geq\alpha_{v_s(F_l)}.\label{minor1}
\end{equation}

Supposons $j_0,\dots,j_{s-1}$ convenables, et construisons $j_s$. Par le lemme \ref{lemsing},
$\dim(\Sing(\{F_l=0\})\cap\{X_0=\dots=X_{v_s(F_l)-1}=0\})\geq s.$ Le th\'eor\`eme de l'intersection projective implique que
$\dim(\Sing(\{F_l=0\})\cap\{X_0=\dots=X_{v_s(F_l)-1}=F_{j_0}=\dots=F_{j_{s-1}}=0\})\geq 0$. Ce ferm\'e est donc non vide ; on y choisit
un point ferm\'e $P$. Comme $\{F_1=\dots=F_c=0\}$ est une intersection compl\`ete lisse,
elle ne peut contenir le point singulier $P$ de $\{F_l=0\}$ : il existe $j_s$ tel que $F_{j_s}$ soit non nul en $P$. 
Comme $F_{j_0},\dots,F_{j_{s-1}}$ s'annulent en $P$,
$j_s\notin\{j_0,\dots,j_{s-1}\}$. Enfin, comme $P\in \{X_0=\dots=X_{v_s(F_l)-1}=0\}$,
on a $\{X_0=\dots=X_{v_s(F_l)-1}=0\}\not\subset \{F_{j_s}=0\}$. 
En consid\'erant les mon\^omes intervenant dans $F_{j_s}$, on voit que cela implique $\deg_{\alpha}(F_{j_s})\geq d_{j_s}\alpha_{v_s(F_l)}$, 
comme voulu.

Soit maintenant $i\in\{1,\dots, c\}$ quelconque. Montrons que : 
\begin{equation}
\frac{\deg_{\alpha}(F_i)}{d_i}\geq-\frac{\sum_{s=0}^{s(F_l)}\alpha_{v_s(F_l)}}{N-s(F_l)}.\label{minor2}
\end{equation}
Si $s(F_i)\geq 0$, cela r\'esulte du lemme \ref{nonalisse} et du choix de $l$. Si $s(F_i)=-1$, on raisonne comme suit.
Par le lemme \ref{positivite}, $\alpha_{v_{s(F_l)}(F_l)}+\frac{\sum_{s=0}^{s(F_l)}\alpha_{v_s(F_l)}}{N-s(F_l)}>0$.
Comme, par le lemme \ref{vFex}, $\alpha_{v_{s(F_l)}(F_l)}$ est le plus petit des $(\alpha_{v_{s}(F_l)})_{0\leq s\leq s(F_l)}$, on en d\'eduit :
$\sum_{s=0}^{s(F_l)}\alpha_{v_s(F_l)}>0$. Par le lemme \ref{alisse}, on a donc 
$\frac{\deg_{\alpha}(F_i)}{d_i}\geq 0\geq-\frac{\sum_{s=0}^{s(F_l)}\alpha_{v_s(F_l)}}{N-s(F_l)}$.

On peut alors conclure. Notons $I=\{j_0,\dots,j_{s(F_l)}\}$ et utilisons les minorations (\ref{minor1})
pour $i\in I$ et (\ref{minor2}) pour $i\notin I$. On obtient :
\begin{equation}
 \sum_{i=1}^c k_i\frac{\deg_{\alpha}(F_i)}{d_i}\geq \sum_{s=0}^{s(F_l)} k_{j_s}\alpha_{v_s(F_l)}
-\Big(\sum_{i\notin I}k_i\Big)\frac{\sum_{s=0}^{s(F_l)}\alpha_{v_s(F_l)}}{N-s(F_l)}.\label{minorons}
\end{equation}
Montrons que le terme de droite de (\ref{minorons}) co\"incide avec :
\begin{equation}
\sum_{s=0}^{s(F_l)}\Big(\alpha_{v_s(F_l)}+\frac{\alpha_{v_0(F_l)}+\dots+\alpha_{v_{s(F_l)}(F_l)}}{N-s(F_l)}\Big)\Big(k_{j_s}-\frac{k_1+\dots+k_c}{N+1}\Big).
\label{miracle}
\end{equation}
Pour cela, on d\'eveloppe (\ref{miracle}), et on identifie les coefficients des $k_i$ avec ceux apparaissant dans le terme de droite de
(\ref{minorons}). Si $i\notin I$, ce coefficient vaut :
\begin{alignat*}{3}
-\frac{1}{N+1}\sum_{s=0}^{s(F_l)}\Big(\alpha_{v_s(F_l)}&+\frac{\alpha_{v_0(F_l)}+\dots+\alpha_{v_{s(F_l)}(F_l)}}{N-s(F_l)}\Big)\\
                         &=-\frac{1}{N+1}\Big(\frac{s(F_l)+1}{N-s(F_l)}+1\Big)\sum_{s=0}^{s(F_l)}\alpha_{v_s(F_l)}\\
                         &=-\frac{1}{N-s(F_l)}\sum_{s=0}^{s(F_l)}\alpha_{v_s(F_l)}.
\end{alignat*}
Si $i=j_s\in I$, un terme suppl\'ementaire appara\^it, de sorte que ce coefficient vaut bien :
$$-\frac{1}{N-s(F_l)}\sum_{s=0}^{s(F_l)}\alpha_{v_s(F_l)}+\alpha_{v_s(F_l)}+\frac{\alpha_{v_0(F_l)}+\dots+\alpha_{v_{s(F_l)}(F_l)}}{N-s(F_l)}=\alpha_{v_s(F_l)}.$$

Il reste \`a montrer que (\ref{miracle}) est positif ou nul. Comme, par le lemme \ref{vFex}, $\alpha_{v_{s(F_l)}(F_l)}$ 
est le plus petit des $(\alpha_{v_{s}(F_l)})_{0\leq s\leq s(F_l)}$,
le lemme \ref{positivite} montre que le premier facteur des termes de la somme (\ref{miracle}) est positif.
Le second facteur des termes de cette somme est
positif ou nul par hypoth\`ese sur les $k_i$. 
\end{proof}

\begin{proof}[$\mathbf{Preuve \text{ }du \text{ }th\acute{e}or\grave{e}me\text{ }\ref{alphadeg}\text{ }(ii)}$]~

  Tout d'abord, en sommant pour $i\in \{1,\dots c\}$ les in\'egalit\'es $k_i>\frac{k_1+\dots+k_c}{N+1}$, 
on montre $(N+1-c)(k_1+\dots+k_c)>0$, donc $k_1+\dots+k_c>0$, et finalement,
$k_i>0$ pour $i\in \{1,\dots, c\}$. 

  On effectue alors la m\^eme preuve que ci-dessus. Le second cas, o\`u il existe $i$ tel que $s(F_i)\geq 0$, est identique :
on obtient une in\'egalit\'e stricte gr\^ace aux hypoth\`eses plus fortes $k_i>\frac{k_1+\dots+k_c}{N+1}$ et \`a l'in\'egalit\'e stricte 
dans le lemme \ref{positivite}.

  Dans le premier cas, o\`u $s(F_i)=-1$ pour tout $i$, on raisonne de m\^eme pour montrer que $\sum_{i=1}^c k_i\frac{\deg_{\alpha}(F_i)}{d_i}\geq 0$. 
Comme $k_i>0$, et par le
cas de stricte in\'egalit\'e du lemme \ref{alisse}, on obtient une in\'egalit\'e stricte sauf \'eventuellement si $d_1=\dots=d_c=2$.
L'\'etude du cas d'\'egalit\'e montre qu'on peut alors supposer
$\deg_{\alpha}(F_1)=\dots=\deg_{\alpha}(F_c)=0$ et $\alpha_i+\alpha_{N-i}=0$ pour $0\leq i\leq N$.

Traitons ce cas directement ; rappelons que par hypoth\`ese, on a alors $c\geq 2$. 
Soit $0\leq r\leq N$ le plus petit entier tel que $\alpha_r>0$. Comme $\alpha_i+\alpha_{N-i}=0$, $r'=N-r$ est le plus grand entier
tel que $\alpha_{r'}<0$. Comme $\deg_{\alpha}(F_i)=0$, on voit que $\{X_0=\dots=X_{r-1}=0\}\subset\{F_i=0\}$, de sorte que $\{X_0=\dots=X_{r-1}=0\}$ est
inclus dans l'intersection compl\`ete $\{F_1=\dots=F_c=0\}$. Montrons qu'il existe un point de $\{X_0=\dots=X_{r-1}=0\}$ en lequel $\{F_1=0\}$ et $\{F_2=0\}$
ont m\^eme espace tangent. Cela contredira la lissit\'e en ce point de $\{F_1=\dots=F_c=0\}$.

Comme $\deg_{\alpha}(F_1)=0$, $\frac{\partial F_1}{\partial X_i}([0:\dots:0:x_r:\dots:x_N])$ est nul si $i>r'$ ; c'est une forme lin\'eaire en
$x_r,\dots,x_N$ si $i\leq r'$. Notons $A_1$ la matrice $(r'+1)\times(r'+1)$ dont les lignes sont les formes lin\'eaires 
$(\frac{\partial F_1}{\partial X_i})_{0\leq i\leq r'}$. De m\^eme, on note $A_2$ la matrice $(r'+1)\times(r'+1)$ dont les lignes sont les formes lin\'eaires 
$(\frac{\partial F_2}{\partial X_i})_{0\leq i\leq r'}$. Consid\'erons $\det(\lambda_1 A_1+\lambda_2 A_2)$ : c'est un polyn\^ome homog\`ene en $\lambda_1$ et
$\lambda_2$. Comme $K$ est alg\'ebriquement clos, on peut trouver $(\lambda_1,\lambda_2)\neq(0,0)$ tels que $\det(\lambda_1 A_1+\lambda_2 A_2)=0$.
Il existe donc $(x_r,\dots, x_N)\neq(0,\dots,0)$ tel que $(\lambda_1 A_1+\lambda_2 A_2)(x_r,\dots, x_N)=0$. 
Alors $\frac{\partial (\lambda_1F_1+\lambda_2 F_2)}{\partial X_i}([0:\dots:0:x_r:\dots:x_N])=0$ pour tout $i$ : c'est ce qu'on voulait.
\end{proof}

\begin{proof}[$\mathbf{Preuve \text{ }du \text{ }th\acute{e}or\grave{e}me\text{ }\ref{alphadeg}\text{ }(iii)}$]~

Montrons l'optimalit\'e dans le cas (i) : supposons donn\'es des r\'eels $k_1,\dots,k_c$ tels que la conclusion de (i) soit satisfaite. Fixons $1\leq j\leq c$.
On choisit $\alpha_0=\dots=\alpha_{N-1}=-1$ et $\alpha_N=N$. Si $i\neq j$, on choisit pour $F_i$ une \'equation g\'en\'erique ne faisant pas intervenir la variable
$X_N$ : en particulier $\deg_{\alpha}(F_i)=-d_i$. 
Par le th\'eor\`eme de Bertini, l'intersection des $\{F_i=0\}_{i\neq j}$ a $[0:\dots:0:1]$ comme unique point singulier. On choisit une \'equation $F_j$
g\'en\'erique, qui \'evite ce point singulier. Le mon\^ome $X_N^{d_j}$ intervient donc dans $F_j$ de sorte que $\deg_{\alpha}(F_j)=N d_j$. De plus,
par le th\'eor\`eme de Bertini, $\{F_1=\dots=F_c=0\}$ est lisse. On peut donc \'ecrire
$$\sum_{i=1}^ck_i\deg_{\alpha}(F_i)=\sum_{i\neq j}-k_i+Nk_j\geq0.$$
Ceci se r\'e\'ecrit $k_j\geq\frac{k_1+\dots+k_c}{N+1}$ comme voulu.

Dans le cas (ii), la m\^eme preuve fonctionne. Il faut seulement v\'erifier qu'il \'etait n\'ecessaire d'exclure le cas $c=1$ et $d_1=2$. 
Pour cela, on prend $\alpha_0=-1$, $\alpha_i=0$ pour $1\leq i\leq N-1$ et $\alpha_N=1$. On choisit alors $F_1=X_0X_N+Q(X_1,\dots,X_{N-1})$ o\`u $Q$ est une 
forme quadratique ordinaire en $X_1,\dots,X_{N-1}$. Alors $\{F_1=0\}$ est lisse, mais $\deg_{\alpha}(F_1)=0$.
\end{proof}

\section{Hilbert-stabilit\'e}\label{hilbstab}

Dans cette partie, on conserve les conventions \ref{notationsgen} et \ref{notationsalpha}, \`a ceci pr\`es qu'on autorise 
$1\leq c\leq N-1$ et $2\leq d_1\leq \dots\leq d_c$.

Dans \cite{Mumfordstab}, Mumford propose de construire des espaces de modules
quasi-projectifs en appliquant la th\'eorie g\'eom\'etrique des invariants
au sch\'ema de Hilbert. On sp\'ecialise ici cette strat\'egie
au cas des intersections compl\`etes, et on fait le
lien avec les r\'esultats des paragraphes pr\'ec\'edents.
On obtient en particulier (Corollaire \ref{conditionhilb}) une condition n\'ecessaire de Hilbert-stabilit\'e des
intersections compl\`etes.

On note $P$ le polyn\^ome de Hilbert des intersections compl\`etes, et
$\Hilb^P_{\mathbb{P}^N}$ le sch\'ema de Hilbert de $\mathbb{P}^N$ correspondant :
c'est un sch\'ema projectif sur $\Spec(\mathbb{Z})$.
Si $Z$ est un sous-sch\'ema de $\mathbb{P}^N_K$ de polyn\^ome de Hilbert $P$ (par exemple une intersection compl\`ete),
on note $[Z]$ le point g\'eom\'etrique de $\Hilb^P_{\mathbb{P}^N}$ correspondant.
Si $l\gg0$, on peut plonger $\Hilb^P_{\mathbb{P}^N}$ dans
la grassmanienne des quotients de dimension $P(l)$ de $H^0(\mathbb{P}^N,\mathcal{O}(l))$, et le fibr\'e de Pl\"ucker induit un fibr\'e
ample sur $\Hilb^P_{\mathbb{P}^N}$ not\'e $\mathcal{P}_l$.
Le sch\'ema en groupes $SL_{N+1}$ agit sur $\Hilb^P_{\mathbb{P}^N}$
par changement de coordonn\'ees, et les fibr\'es en droites $\mathcal{P}_l$
sont naturellement lin\'earis\'es.
On dit que $Z$ est Hilbert-stable si $[Z]\in(\Hilb^P_{\mathbb{P}^N})^s(\mathcal{P}_l)$ pour $l\gg0$.
Par le crit\`ere de Hilbert-Mumford,
il est \'equivalent de demander que, pour $l\gg0$, pour tout sous-groupe \`a un
param\`etre  $\rho : \mathbb{G}_{m,K}\to SL_{N+1,K}$, on ait
$\mu^{\mathcal{P}_l}([Z],\rho)>0$. Pour pouvoir
utiliser ce crit\`ere, explicitons ces fonctions $\mu$ dans notre situation.

\begin{prop}\label{muhilbeg}
Soit $l\gg0$, $Z$ une intersection compl\`ete sur
le corps al\-g\'e\-bri\-quement clos $K$ et $\rho : \mathbb{G}_{m,K}\to SL_{N+1,K}$
un sous-groupe \`a un param\`etre choisi comme 
dans le lemme \ref{bon1ps}.
Alors :
\begin{equation*}
\mu^{\mathcal{P}_l}([Z],\rho)=\min_{\mathfrak{B}}\Big(\sum_{F\in\mathfrak{B}}\deg_{\alpha}(F)\Big),
\end{equation*}
o\`u le $\min$ porte sur les bases $\mathfrak{B}$ de $H^0(\mathbb{P}^N_K,\mathcal{I}_Z(l))$.
\end{prop}

\begin{proof}[$\mathbf{Preuve}$]~
Ce calcul classique se trouve par exemple dans \cite{HarrisMorrison} Prop. 4.23.
\end{proof}

\subsection{Majoration des fonctions $\mu$}

On majore ici la quantit\'e $\mu^{\mathcal{P}_l}([Z],\rho)$ calcul\'ee dans la proposition \ref{muhilbeg}.

\begin{lemme}\label{muhilbineg}
Soient $Z=\{F_1=\dots=F_c=0\}$ une intersection compl\`ete sur le corps alg\'ebriquement clos $K$,
$\rho : \mathbb{G}_{m,K}\to SL_{N+1,K}$ un sous-groupe \`a un param\`etre choisi comme 
dans le lemme \ref{bon1ps} et $l\gg0$.
Alors :
\begin{equation}\label{eqmuhilb}
\mu^{\mathcal{P}_l}([Z],\rho)\leq\sum_{i=1}^c\deg_\alpha(F_i)
\Big(\sum_{i\notin I\subset\{1,\dots,c\}}(-1)^{c-1-|I|}\tbinom{N+l-\sum_{j\notin I}d_j}{N}\Big).
\end{equation}
\end{lemme}

\begin{proof}[$\mathbf{Preuve}$]~
Si $I\subset\{1,\dots,c\}$, on note $V^I_l=H^0(\mathbb{P}^N_K,\mathcal{O}(l-\sum_{i\notin I} d_i))$. 
Pour $0\leq r\leq c$, on pose $K^r_l=\bigoplus\limits_{\substack{I\subset\{1,\dots,c\}\\|I|=r}}V^I_l$.
Consid\'erons la r\'esolution de Koszul de $\mathcal{O}_Z$ sur $\mathbb{P}^N_K$. Si l'on tensorise cette r\'esolution par $\mathcal{O}(l)$ 
pour $l\gg0$, le complexe obtenu en prenant les sections globales reste exact par annulation de Serre. 
On obtient ainsi une suite exacte longue de la forme :
\begin{equation*}
 0\to K^0_l\to\dots\stackrel{d^{r-1}}{\rightarrow}K^r_l\stackrel{d^r}{\rightarrow}\dots\to K^c_l\to H^0(Z,\mathcal{O}(l))\to 0,
\end{equation*}
o\`u  $d^r:V^I_l\to V^J_l$ est au signe pr\`es la multiplication par $F_i$ si $J=I\cup\{i\}$ et est nul dans les autres cas.
On notera $N^r_l=\Ker(d^r)=\Ima(d^{r-1})\subset K^r_l$. 

On introduit les notations suivantes.
Si $F\in V^I_l$, on pose : $$\deg'_\alpha(F)=\deg_\alpha(F)-\sum_{i\notin I}\deg_\alpha(F_i).$$ 
Si $\Phi=(F_I)\in K^r_l$, on pose $\deg'_\alpha(\Phi)=\max_I\deg'_\alpha(F_I)$.
 De plus, on modifie l\'eg\`erement les conventions \ref{notationsalpha} : dans toute cette preuve, la notation $\Phi^\alpha$ ou la notion
d'\'el\'ement $\alpha$-homog\`ene fait r\'ef\'erence \`a $\deg_\alpha'$ et non \`a $\deg_\alpha$. Remarquons que si $0\leq r<c$ et $\Phi\in K^r_l$,
par d\'efinition de $\deg_{\alpha}'$ et vu l'expression de $d^r$, on a $\deg'_\alpha(d^r(\Phi))\leq\deg'_\alpha(\Phi)$.

On va montrer par r\'ecurrence sur $0\leq r\leq c$ l'\'enonc\'e suivant : il existe une base $\mathfrak{B}^r_l$ de $N^r_l$
telle que : 
\begin{equation}\label{hyprec}
\sum_{\Phi\in\mathfrak{B}^r_l}\deg_{\alpha}'(\Phi)\leq \sum_{i=1}^c\deg_\alpha(F_i)
\Big(\sum\limits_{\substack{i\notin I\subset\{1,\dots,c\}\\|I|\leq r-1}}(-1)^{r-1-|I|}\dim(V_l^I)\Big).
\end{equation}
Pour $r=0$, $N^0_l=\{0\}$, de sorte qu'on peut prendre $\mathfrak{B}^0=\varnothing$.

Supposons l'\'enonc\'e vrai pour $r$ et montrons-le pour $r+1$. Pour cela, soit $\mathfrak{B}^r_l$ une base de $N^r_l$ telle que 
$\sum_{\Phi\in\mathfrak{B}^r_l}\deg_{\alpha}'(\Phi)$ soit minimal. On voit ais\'ement
que $\mathfrak{B}^{r,\alpha}_l=\{\Phi^\alpha,\Phi\in\mathfrak{B}^r_l\}$ est une famille libre d'\'el\'ements
$\alpha$-homog\`enes de $K^r_l$. Compl\'etons cette famille en une base $\mathfrak{C}^{r}_l$ de $K_l^r$ constitu\'ee d'\'el\'ements 
$\alpha$-homog\`enes. Remarquons que $\sum_{\Phi\in\mathfrak{C}}\deg'_\alpha(\Phi)$
ne d\'epend pas de la base $\mathfrak{C}$ de $K_l^r$ constitu\'ee d'\'el\'ements 
$\alpha$-homog\`enes.  Utilisant $\alpha_0+\dots+\alpha_N=0$, cette quantit\'e est facile \`a calculer pour la base
$\mathfrak{C}$ constitu\'ee des mon\^omes.
Il vient donc :
\begin{equation}\label{basemonome}
\sum_{\Phi\in\mathfrak{C}^{r}_l}\deg'_\alpha(\Phi)=\sum_{\Phi\in\mathfrak{C}}\deg'_\alpha(\Phi)=\sum\limits_{\substack{I\subset\{1,\dots,c\}\\|I|=r}}
\dim(V_l^I)\Big(-\sum_{i\notin I}\deg_\alpha(F_i)\Big).
\end{equation}
Comme $\{\Phi^\alpha,\Phi\in\mathfrak{B}^r_l\cup(\mathfrak{C}^{r}_l\setminus\mathfrak{B}^{r,\alpha}_l)\}=\mathfrak{C}^{r}_l$ est une base
de $K_l^r$, $\mathfrak{B}^r_l\cup(\mathfrak{C}^{r}_l\setminus\mathfrak{B}^{r,\alpha}_l)$ est \'egalement une base de $K_l^r$. En particulier,
$\mathfrak{C}^{r}_l\setminus\mathfrak{B}^{r,\alpha}_l$ est une base d'un suppl\'ementaire de $N^r_l$ dans $K^r_l$, de sorte
que $\mathfrak{B}^{r+1}_l=d^r(\mathfrak{C}^{r}_l\setminus\mathfrak{B}^{r,\alpha}_l)$ est une base de $N^{r+1}_l$.
Montrons que cette base convient. Pour cela,
on calcule :
\begin{alignat*}{3}
\sum_{\Phi\in\mathfrak{B}^{r+1}_l}\deg'_\alpha(\Phi)&\leq\sum_{\Phi\in(\mathfrak{C}^{r}_l\setminus\mathfrak{B}^{r,\alpha}_l)}\deg'_\alpha(\Phi) 
% \\                                                     &
=\sum_{\Phi\in\mathfrak{C}^{r}_l}\deg'_\alpha(\Phi)-\sum_{\Phi\in\mathfrak{B}^{r}_l}\deg'_\alpha(\Phi)\\
                                                     &\leq\sum_{i=1}^c\deg_\alpha(F_i)
\Big(\sum\limits_{\substack{i\notin I\subset\{1,\dots,c\}\\|I|\leq r}}(-1)^{r-|I|}\dim(V_l^I)\Big),
\end{alignat*}
o\`u l'on a utilis\'e respectivement (\ref{basemonome}) et l'hypoth\`ese de r\'ecurrence (\ref{hyprec}) pour \'evaluer les deux termes. Cela conclut la r\'ecurrence.

Faisons \`a pr\'esent $r=c$ dans (\ref{hyprec}). On obtient une base $\mathfrak{B}_l^c$ de 
$N^c_l=\Ker[H^0(\mathbb{P}^N_K,\mathcal{O}(l))\to H^0(Z,\mathcal{O}(l))]=H^0(\mathbb{P}^N_K,\mathcal{I}_Z(l))$ car $l\gg0$. 
Comme $\deg_\alpha$ et $\deg'_\alpha$ co\"incident pour
des \'el\'ements de $H^0(\mathbb{P}^N_K,\mathcal{O}(l))$, et comme $\dim(V^I_l)=\tbinom{N+l-\sum_{j\notin I}d_j}{N}$, il vient :
$$\sum_{F\in\mathfrak{B}_l^c}\deg_{\alpha}(F)\leq\sum_{i=1}^c\deg_\alpha(F_i)
\Big(\sum_{i\notin I\subset\{1,\dots,c\}}(-1)^{c-1-|I|}\tbinom{N+l-\sum_{j\notin I}d_j}{N}\Big).$$
Par la proposition \ref{muhilbeg}, cela conclut.
\end{proof}

On en d\'eduit la proposition suivante : 

\begin{prop}\label{muhilbasympt}
Soient $Z=\{F_1=\dots=F_c=0\}$ une intersection compl\`ete sur le corps alg\'ebriquement clos $K$
et $\rho : \mathbb{G}_{m,K}\to SL_{N+1,K}$ un sous-groupe \`a un param\`etre choisi comme 
dans le lemme \ref{bon1ps}.
Alors :
\begin{equation}\label{eqmuhilb2}
\limsup_{l\to+\infty} \frac{\mu^{\mathcal{P}_l}([Z],\rho)}{l^{N-c+1}}\leq\frac{d_1\dots d_c}{(N-c+1)!}\sum_{i=1}^c\frac{\deg_{\alpha}(F_i)}{d_i}.
\end{equation}
\end{prop}

\begin{proof}[$\mathbf{Preuve}$]
Le polyn\^ome $\sum_{i\notin I\subset\{1,\dots,c\}}(-1)^{c-1-|I|}\tbinom{N+X-\sum_{j\notin I}d_j}{N}$ 
a pour terme dominant $\frac{d_1\dots\hat{d_i}\dots d_c}{(N-c+1)!}X^{N-c+1}$, comme le montre une r\'ecurrence sur $c$.

Ainsi, le terme de droite dans l'in\'egalit\'e (\ref{eqmuhilb}) est un polyn\^ome en $l$ de degr\'e $\leq N-c+1$ et dont le coefficient 
de $l^{N-c+1}$ est $\frac{d_1\dots d_c}{(N-c+1)!}\sum_{i=1}^c\frac{\deg_{\alpha}(F_i)}{d_i}$.
On conclut en divisant par $l^{N-c+1}$ l'in\'egalit\'e (\ref{eqmuhilb}), et en faisant tendre $l$ vers $+\infty$.
\end{proof}

\subsection{Condition n\'ecessaire de Hilbert-stabilit\'e}\label{hilbstabdiscussion}

La proposition \ref{muhilbasympt} a pour corollaire imm\'ediat une condition n\'ecessaire de Hilbert-stabilit\'e, qui
est le r\'esultat principal de cette partie.
\begin{cor}\label{conditionhilb}
 Soit $Z=\{F_1=\dots=F_c=0\}$ une intersection compl\`ete Hilbert-stable sur le corps alg\'ebriquement clos $K$. Alors,
si $\alpha_0\leq\dots\leq\alpha_N$ sont des entiers de somme nulle, et quelque soit le
syst\`eme de coordonn\'ees choisi,
\begin{equation}\label{cnhilb}
\sum_{i=1}^c\frac{\deg_{\alpha}(F_i)}{d_i}\geq 0.
\end{equation}
\end{cor}
Le Theorem 1.1 de \cite{YujiSano} montre que si $Z$ est Chow-stable, l'in\'ega\-li\-t\'e (\ref{cnhilb}) est stricte.
Rappelons que, par un th\'eor\`eme de Fogarty (\cite{Fogarty}, voir aussi \cite{GIT} App. 4C), la Chow-stabilit\'e de $Z$ implique la
Hilbert-stabilit\'e de $Z$ (en g\'en\'eral, on n'a pas l'implication inverse).
Le corollaire \ref{conditionhilb} et le Theorem 1.1 de \cite{YujiSano} sont donc tr\`es proches mais
ne peuvent se d\'eduire l'un de l'autre.

Enfin, l'article \cite{YujiSano} affirme (c'est la preuve du Corollary 1.2)
que, si $c=2$, et si l'in\'egalit\'e (\ref{cnhilb}) est v\'erifi\'ee et est stricte, $Z$ est Hilbert-stable. L'argument donn\'e est
malheureusement erron\'e.

\vspace{1em}

 Quand $Z$ est lisse, l'in\'egalit\'e (\ref{cnhilb}) est vraie par le th\'eor\`eme
\ref{alphadeg} pour $k_1=\dots=k_c=1$. 
Autrement dit, le th\'eor\`eme \ref{alphadeg} implique une forme faible de la Hilbert-stabilit\'e
des intersections compl\`etes lisses.

La Hilbert-stabilit\'e des intersections compl\`etes lisses est connue dans tr\`es peu de cas.
Quand $c=1$ est trivial, $\Hilb^P_{\mathbb{P}^N}$
est un espace projectif, tous les fibr\'es amples $\mathcal{P}_l$ introduits ci-dessus sont donc
n\'ecessairement proportionnels au fibr\'e $\mathcal{O}(1)$, et les hypersurfaces
lisses sont Hilbert-stables par \cite{GIT} Prop. 4.2.
Signalons un cas non trivial o\`u la Hilbert-stabilit\'e est connue.
 Quand $N=3$, $c=2$, $d_1=2$ et $d_2=3$, Casalaina-Martin, Jensen et Laza
\cite{Lazaetcie} montrent que les intersections
compl\`etes lisses sont Chow-stables, ce qui implique leur
Hilbert-stabilit\'e par le th\'eor\`eme de Fogarty
mentionn\'e ci-dessus. 

Pour montrer la Hilbert-stabilit\'e d'une intersection compl\`ete lisse $Z$, la difficult\'e
suppl\'ementaire par rapport au th\'eor\`eme \ref{alphadeg} est une estimation de la diff\'erence entre les deux termes de l'in\'egalit\'e
(\ref{eqmuhilb}), qui d\'epend des compensations entre termes de $\alpha$-degr\'e maximal des \'equations de $Z$.
Comme le signale le rapporteur,
il serait d\'ej\`a int\'eressant de montrer la Hilbert-stabilit\'e d'une intersection compl\`ete g\'en\'erale,
dans l'esprit de \cite{Alpergen}.
%\vspace{1em}
%
%Quand $d_1<d_2=\dots=d_c$, on peut faire un lien avec la partie \ref{partieqp}.
%Par la proposition \ref{mucombo}, le terme de gauche
%de (\ref{cnhilb})
%est exactement la quantit\'e intervenant dans
%crit\`ere de Hilbert-Mumford pour $Z$ relativement au fibr\'e en droites $SL_{N+1}$-lin\'earis\'e
%$\mathcal{O}(d_2,d_1)$ sur $\bar{H}$.
%
%  Ceci s'explique de la mani\`ere suivante. Le fibr\'e en droites $\mathcal{P}_l|_{H^{ic}}$ vu comme
%\'el\'ement de $\Pic(\bar{H})\simeq\Pic(H^{ic})$ 
%est de la forme $\mathcal{O}(\lambda_{1,l},\lambda_{2,l})$ et
%$\lim_{l\to\infty}\frac{\lambda_{1,l}}{\lambda_{2,l}}=\frac{d_2}{d_1}$.
%
% Cependant, par la proposition \ref{ample}, $\mathcal{O}(d_2,d_1)$ n'est
%jamais ample sur $\bar{H}$, de sorte qu'on ne peut pas appliquer
%le crit\`ere de Hilbert-Mumford. En un certain sens, la positivit\'e
%du terme de gauche de (\ref{cnhilb}) pour $Z$ lisse n'est pas assez forte pour impliquer la quasi-projectivit\'e
%de l'espace de modules grossier $M$.
\addcontentsline{toc}{section}{R\'{e}f\'{e}rences}

\bibliographystyle{plain-fr}
\bibliography{biblio}

\end{document}